\documentclass{amsart}

\usepackage{amsmath, amssymb, amsthm, amsrefs, pst-node, tikz-cd, mathtools}
\usepackage[left=3cm, right=3cm, top=3cm]{geometry}

\newtheorem{theorem}{Theorem}[section]
\newtheorem{lemma}[theorem]{Lemma}
\newtheorem{proposition}[theorem]{Proposition}
\theoremstyle{definition}
\newtheorem{definition}[theorem]{Definition}

\theoremstyle{remark}

\numberwithin{equation}{section}

\newcommand{\p}{\mathfrak{p}}
\newcommand{\fp}{\mathbb{F}_p}
\newcommand{\fpbar}{\overline{\mathbb{F}}_p}

\newcommand{\qp}{\mathbb{Q}_p}
\newcommand{\cp}{\mathbb{C}_p}

\newcommand{\ord}{\text{ord}}
\newcommand{\cm}{\mathcal{CM}}

\newcommand{\okw}{O_{K_{0,v_0}}}
\newcommand{\okv}{O_{K_v}}

\newcommand{\Z}{\mathcal{Z}}

\newcommand{\Hom}{\text{Hom}}

\newcommand{\D}{\mathbb{D}}

\newcommand{\ktildebarp}{\overline{\tilde{k}}_\p}

\begin{document}

\title{Special Correspondences of CM Abelian Varieties and Eisenstein Series II}

\author{Ali Cheraghi}
\address{The Fields Institute, 222 College Street, Toronto, ON M5T 3J1}
\email{acheragh@fields.utoronto.ca}




\setlength{\parindent}{0cm}

\begin{abstract}
In this paper, we prove the relation between special cycles on a Rapoport-Smithling-Zhang Shimura variety and special values of the derivative of a Hilbert Eisenstein series. 
\end{abstract}

\maketitle
\tableofcontents

\section{Introduction}

This is a sequel to the \cite{cheraghi}. In that paper, we considered one fixed CM field $K$ and considered pairs of CM principally polarized abelian varieties which have CM-types that are different in only one embedding. Considering their moduli space, we defined special divisors that depend on $\alpha$ where $\alpha$ is an element of $F$ (the maximal totally real field inside $K$) and computed their Arakelov degrees by calculating the number of stacky points of special fibers of the special divisors and also the length of strictly Henselian local rings (c.f. \cite{cheraghi} theorem 3.10 and 3.11) and putting these two together. On the other hand we found an Eisenstein series for which the $\alpha^{th}$ Fourier coefficient was related (up to some factors that do not depend on $\alpha$) to the Arakelov degree of $\alpha^{th}$ special divisor. The main theorem of \cite{cheraghi} was theorem 5.2 that showed this relation. 

In this part, we are interested in having different but included CM fields (i.e. CM fields $K_0$ and $K$ with $K_0 \subseteq K$) and then we are going to consider pairs of polarized abelian varieties (such that their dimensions are relatively $[K:K_0]$ and they have action by $O_{K_0}$ (ring of integers of $K_0$) such that the action of $O_{K_0}$ on their Lie algebras has a specific kind. Then using the same method as in our previous paper (which used a method originally from \cite{howard}), we define special divisors and prove that their Arakelov degrees are related to the Fourier coefficients of an Eisenstein series. The main motivation for these kinds of results for the author is the expected relation of the $0^{th}$ coefficient of this Eisenstein series to special value of the derivative of L-functions.

To state our main result, we need some notations (we will repeat these notations in the notations section below as well). Let $K_0 \subseteq K$ be CM-fields with $F_0 \subseteq F$ being their maximal totally real subfields. Let $\Phi_0$ and $\Phi$ be some nearby CM-types (for the precise definition, see the notations section) of $K_0$ and $K$, respectively. Let $\tilde{K}$ be the reflex field of $(K,\Phi)$. For a prime $\p$ of $\tilde{K}$, let $\overline{\tilde{k}_{\p}}$ be a choice of algebraic closure of residue field of $\tilde{K}$ at $\p$. Let $A_0$ be a principally polarized abelian variety with CM by $O_{K_0}$ with polarization $\lambda_{A_0}$ and let $A$ be a polarized abelian variety with polarization $\lambda_A$ that has an action by $O_{K_0}$. We can make $\Hom_{O_{K_0}}(A_0,A)$ a Hermitian space by letting $\langle f,g \rangle = \lambda_{A_0}^{-1} \circ g^{\vee } \circ \lambda_A \circ f$ where $g^{\vee}: A^{\vee} \rightarrow A_0^{\vee}$ is the dual of $g$. This Hermitian form is $O_{K_0}$-valued and we define $\langle , \rangle_{CM}$ to be the unique $K$-valued Hermitian form with the property that tr$\langle , \rangle_{CM} =  \langle , \rangle$. We define the special divisors $\mathcal{Z}(\alpha)$ and consider the following quantity (called Arakelov degree of this special divisor on specific moduli space):
$$\widehat{\deg} \Z(\alpha) = \frac{1}{[\tilde{K}:\mathbb{Q}]} \sum_{\p \subset O_{\tilde{K}}} \log\; N(\p) \sum_{z \in \Z(\alpha)(\overline{\tilde{k}_{\p}})} 
\frac{\text{length}(O^{\text{\'et}}_{\Z(\alpha),z})}{ \# \text{Aut}\; z} $$ 

where $O^{\text{\'et}}_{\Z(\alpha),z}$ is the strictly Henselian local ring at $z$ and $N(\p)$ is the norm of $\p$ in$\mathbb{Q}$ and $O_{\tilde{K}}$ is the ring of integers of $\tilde{K}$.

In the previous paper, we were able to compute this quantity for the special divisors on a related moduli space for all nonzero $\alpha$  and then considered an Eisenstein series which has $b_{\Phi}(\alpha,y)$ as its $\alpha^{th}$ Fourier coefficient where $\tau = x+iy$ is an element of $\mathbb{H}^{[F:\mathbb{Q}]}$ (where $\mathbb{H}$ is the upper half-plane) and the main result of \cite{cheraghi} was the following:

\begin{theorem}
Let $\alpha$ be a nonzero element of $F$. Suppose that the following ramification conditions are satisfied:\\
1) $K/F$ is ramified at at least one finite prime.

2) For every rational prime $l \leq \frac{[\tilde{K}:\mathbb{Q}]}{[K:\mathbb{Q}]}+1$, the ramification index of $l$ in $\tilde{K}$ is less than $l$,
then we have
$$\widehat{\deg}{\mathcal{Z}}(\alpha) = \frac{-|C_K|}{w(K)} \frac{\sqrt{N_{F/\mathbb{Q}}(d_{K/F})}}{2^{r-1} [K:\mathbb{Q}]} b_{\Phi}(\alpha,y)$$
where $|C(K)| = |\widehat{O}_F^{\times \gg 0}/N_{K/F}\widehat{O}_K^{\times}| h(K)$ where $h(K)$ is the class number of $K$, $w(K)$ is the number of roots of unity in $K$, $d_{K/F}$ is the relative discriminant of $K/F$, $r$ is the number of places (including archimedean) ramified in $K$, and $b_{\Phi}(\alpha,y)$ is the  $\alpha^{\text{th}}$ coefficient of the Fourier expansion of the derivative of a Hilbert Eisenstein series at $s=0$.
\end{theorem}

In this paper, we prove the same result for $\mathcal{Z}(\alpha)$ and moduli space in the setting we wrote about above. Specifically, we find an Eisenstein series with Fourier coefficients denoted by $b_{\Phi}(\alpha,y)$ and prove the following {\bf main theorem}:

\begin{theorem}
Let $\alpha \in F^{\times}$. Suppose that the following conditions are satisfied:\\
(1) $K/F$ is ramified at at least one finite prime.\\
(2) Relative discriminants of $K_0/F_0$ and $F/F_0$ are relatively prime.\\
(3) The assumption below proposition 4.1 below is satisfied, then we have:
$$\widehat{\deg}\mathcal{Z}(\alpha) = \frac{-1}{w(K_0)}\frac{N_{F/\mathbb{Q}}(d_{K/F})^{\frac{1}{2}}}{2^{r-1}[K:\mathbb{Q}]} b_{\Phi}(\alpha,y).$$
\end{theorem}

The way to prove it is to consider CM p-divisible groups in the third section and prove the amount of lifting of homomorphisms between CM p-divisible groups, and then in the fourth section we are going to consider the global case and define the moduli space and the special divisors as DM-stacks. Finally in the fifth and last chapter, we will define the Eisenstein series and prove the relation between the Fourier coefficients and the Arakelov degree which will resolve the main theorem.

\section{Notations}

Let $p$ be a prime number. Let $K_0 \subseteq K$ be CM fields with $F_0 \subseteq F$ their maximal totally real subfields. Let $K_0 = F_0(\sqrt{\Delta})$ for a totally negative element $\Delta \in F_0$. For a local or global field $L$, let $O_L$ be the ring of integers or valuation ring of $L$. For a global field $L$, let  $O_{L,(v)}$ be the localisation of ring of integers of $L$ at $v$ and for $v$ a prime of a number field $L$, $L_v$ is the completion of $L$ at $v$ and $O_{L,v}$ the valuation ring of $L_v$. Let $\fpbar$ be an algebraic closure of $\fp$, the field of $p$ elements. $\bar{L}_v$ is a choice of algebraic closure of $L_v$.  For a local field $L$, $\p_L$ denotes the maximal ideal of $O_L$. For a prime $\mathfrak{q}$ of $F$, let $\epsilon_{\mathfrak{q}}$ be $0$ if $\mathfrak{q}$ is ramified in $K$ and $1$ if $\mathfrak{q}$ is unramified in $K$. For a prime $v$ of $K$, $v_0$ a prime of $K_0$ with $v|v_0|p$. Let $v_F$, $v_{0,F_0}$ be primes in $F$, $F_0$ respectively below $v$, $v_0$. Let $n=[K:K_0]$ and $2d = [K_0:\mathbb{Q}]$. Let $\chi: \mathbb{A}_F^{\times} \rightarrow \lbrace \pm 1 \rbrace$ be the quadratic character associated to $K/F$. By a CM-type $\Phi_L$ of a CM-field $L$ with complex conjugation $\bar{(.)}$, we mean a subset of $\Hom(L,\mathbb{C})$ such that $\Phi_L \coprod \overline{\Phi_L} = \Hom(L,\mathbb{C})$, where $\overline{\Phi_L} = \lbrace \bar{\phi}| \phi \in \Phi_L \rbrace$ ($\bar{\phi}(x) = \phi(\bar{x})$ for $x \in L$). For a finite extension $L/\qp$ and subfield $L_0 \subseteq L$ of index 2 and $\text{Gal}(L/L_0) = \langle\bar{(.)}\rangle$, a $p$-adic CM-type $\Phi_L$ is a subset of $\Hom_{\qp}(L,\cp)$ with $\Phi_L \coprod \overline{\Phi_L} = \Hom_{\qp}(L,\cp)$ where  $\overline{\Phi_L} = \lbrace \bar{\phi}| \phi \in \Phi_L \rbrace$ ($\bar{\phi}(x) = \phi(\bar{x})$ for $x \in L$). Let $\Phi$ (resp. $\Phi_0$) be a CM-type of $K$ (resp. $K_0$) with 
$$\Phi = \lbrace \phi_1^1, \phi_1^2, \cdots, \phi_1^n, \phi_2^1, \cdots, \phi_2^n, \cdots, \phi_d^1, \cdots, \phi_d^n \rbrace$$
$$\Phi_0 = \lbrace \phi_1, \phi_2, \cdots, \phi_d \rbrace$$
with $\phi_i^j|_{K_0} = \phi_i$ if $(i,j) \neq (1,1)$ and $\phi_1^1|_{K_0} = \overline{\phi_1}$. Also let $\tilde{\Phi}_0$ be the CM-type of $K$ induced by $\Phi_0$.
Fix $\iota: \mathbb{C} \cong \cp$.
Let $\Phi_v$ (resp. $\Phi_{v_0}$) be $p$-adic CM-type of $K_v$ relative to $F_{v_F}$ (resp. $K_{0,v_0}$ relative to $F_{0,v_{F_0}}$) consisting of all $\phi \in \Phi$ (resp. $\phi_0 \in \Phi_0$) with the property that $(\iota \circ \phi)^{-1}(\p_{\cp}) = v$ (resp. $(\iota \circ \phi_0)^{-1}(\p_{\cp}) = v_0$). Let $\tilde{K} \subseteq \mathbb{C}$ (resp.  $\tilde{K}_p \subseteq \cp$ with the abuse of notation) be a large enough Galois extension of $\mathbb{Q}$ (resp. $\qp$) such that for $\sigma \in \text{Aut}(\mathbb{C}/\tilde{K})$ (resp. $\text{Aut}(\cp/\tilde{K}_p))$, we have $\Phi^\sigma = \Phi$ (resp. $\Phi_v^\sigma = \Phi_v$) and $\Phi_0^\sigma = \Phi_0$ (resp. $\Phi_{0,v_0}^\sigma = \Phi_{0,v_0}$) where $\Phi^\sigma = \lbrace \sigma \circ \phi | \phi \in \Phi \rbrace$ and similarly for $\Phi_0$ (For example, we can take $\tilde{K}$ (resp. $\tilde{K}_p$) to be the Galois closure of $K$ over $\mathbb{Q}$ (resp. $K_v$ over $\qp$)). Let $\tilde{k}_{\p}$ be the residue field of $\tilde{K}$ at $\p$. Let $\ktildebarp$ be an algebraic closure of $\tilde{k}_{\p}$. $\tilde{W}$ be the valuation ring of the maximal unramified extension of $\tilde{K}_p$ and $m$ be its maximal ideal. ART be the category of local Artinian $\tilde{W}$-algebras with residue field $\fpbar$. For a $p$-divisible group $A$ defined over $R \in \text{obj(ART)}$ with an action $\kappa: O_K \rightarrow \text{End}(A)$, we say it has $\Phi$-determinant condition if determinant of the action of $\sum_{i=1}^r t_ix_i$ ($x_i$'s $\in O_K$ and $t_i$'s variables) on $\text{Lie}\; A$ is given by the image of $\prod_{\phi \in \Phi}(\sum_{i=1}^r t_i \phi(x_i))$ in $R[t_1, \cdots, t_r]$. For $R \in \text{obj(ART)}$, we let $m_R$ be the maximal ideal of $R$, then we denote $R^{(n)} = R/m_R^n$. Let $J_{\Phi_v}$ (resp. $J_{0,\Phi_{0,v_0}}$) be the kernel of the $\tilde{W}$-algebra map 
$$O_{K_v} \otimes_{\mathbb{Z}_p} \tilde{W} \rightarrow \prod_{\phi \in \Phi_v} \cp$$ (resp. $O_{K_0,v_0} \otimes_{\mathbb{Z}_p} \tilde{W} \rightarrow \prod_{\phi \in \Phi_{v_0}} \cp$) given by $x \otimes 1 \mapsto (\phi(x))_{x \in \Phi_v}$ (resp. $x \otimes 1 \mapsto (\phi(x))_{x \in \Phi_{v_0}}$). $\mathcal{D}_v$ and $\mathcal{D}_{v_0}$ be differents of $K_v/\qp$ and $K_{0,v_0}/\qp$, respectively. Let $W$ be ring of integers of maximal unramified extension of $K_v$ if $v$ is known in the context.

$\mathcal{D}_0$, $\mathcal{D}$ be the differents of $K_0/\mathbb{Q}$ and $K/\mathbb{Q}$ respectively. For two number fields $L_1 \subseteq L_2$, $\partial_{L_2/L_1}$ be the relative different of $L_2$ over $L_1$. We assume that $K_0/F_0$ is ramified at at least one finite prime and the relative discriminants of $K_0/F_0$ and $F/F_0$ are relatively prime (this is to ensure the existence of CM abelian varieties with $O_K$-action and $O_{K_0}$-action). Also for two abelian varieties $A_0$ and $A$ with CM by $O_{K_0}$, let $L(A_0,A)$ be $\Hom_{O_{K_0}}(A_0,A)$ ($O_{K_0}$-linear mappings from $A_0$ to $A$).

\section{Local part}

\subsection{Lifting of homomorphisms}

 We assume $O_{K_0} \otimes_{\mathbb{Z}} O_F = O_K$, also we assume the following ramification condition:
 
If $p \leq \frac{[\tilde{K}_p:\qp]}{[K_v:\qp]} + 1$, then ramification index of $\tilde{K}_p/\qp$ is less than $p$. 
 
  Let $v$ (resp. $v_0$) be a prime of $K$ (resp. $K_0$) over $p$ such that $v | v_0$ and $A$ (resp. $A_0$) be a $p$-divisible group over $\fpbar$ with an action by $O_{K_v}$ (resp. $O_{K_{0, v_0}}$) given by $\kappa: O_{K_v} \rightarrow \text{End}\; A$ (resp. $\kappa_0: O_{K_{0,v_0}} \rightarrow \text{End}\; A_0$)
having $\Phi_v$-determinant (resp. $\Phi_{0,v_0}$-determinant) condition. Also we assume that they have an $O_{K_v}$-linear (resp. $O_{K_{0,v_0}}$-linear) polarization $\lambda: A \rightarrow A^{\vee}$ with kernel $A[\mathfrak{a}]$ where $\mathfrak{a}$ is an ideal of $O_{K_v}$(resp. principal polarization $\lambda_{0}: A_0 \rightarrow A_0^{\vee}$). Now we  consider these two cases: 

$1.$ All elements of $\Phi_v$ restricted to $K_{0,v_0}$ become elements of $\Phi_{0,v_0}$.

$2.$ There's exactly one element of $\Phi_v$ such that when restricted to $K_{0,v_0}$ becomes conjugate of an element of $\Phi_{0,v_0}$ and $K_v \neq F_{v_F}$. 

First we assume we have case 1:

\begin{proposition}
Let $T \in \text{obj}(ART)$ and $(A_0^{\prime},\kappa_0^{\prime}, \lambda_0^{\prime})$, $(A^{\prime},\kappa^{\prime},\lambda^{\prime})$ be the unique deformations of $(A_0,\kappa_0,\lambda_0)$ and $(A,\kappa,\lambda)$ to $T$ (which exist by theorem 2.1.3 of \cite{howard}). The reduction map
$$\text{Hom}_{O_{K_{0,v_0}}}(A_0^{\prime},A^{\prime}) \rightarrow \text{Hom}_{O_{K_{0,v_0}}}(A_0,A)$$ 
is a bijection.
\end{proposition}

\begin{proof}
If $g: S \twoheadrightarrow R$ is a surjection 
in $\text{ART}$ with kernel ${\ker g}$ 
having property $(\ker g)^2 = 0$. Denote by $
(A^{R},\kappa^R,\lambda^R)$ the deformation of 
$(A,\kappa,\lambda)$ to $R$ and similarly for $
(A_0,\kappa_0,\lambda_0)$. Now assume that we 
have $f \in \text{Hom}_{O_{K_{0,v_0}}}
(A_0^R,A^R)$ and let $\D_{A_0^R}$ and $\D_{A^R}
$ be the Grothendieck-Messing crystals of 
$A_0^R$, $A^R$ respectively. $f$ induces a map  $f: \D_{A_0^R}(S) \rightarrow \D_{A^R}(S)$. Now as we are in case 1, we have 
$$J_{0,\Phi_{0,v_0}}(\okv \otimes_{\mathbb{Z}_p} \tilde{W}) \subseteq J_{\Phi_v}$$
and so 
$$f(J_{0,v_0}\D_{A_0^R}(S)) = J_{0,v_0}(\D_{A_0^R}(S)) \subseteq J_v\D_{A_0^R}(S).$$
By the proof of theorem 2.1.3 in \cite{Howard2012}, Hodge filtrations of the deformations to $S$ correspond to
$$J_{0,\Phi_{0,v_0}}\D_{A_0^R}(S) \subseteq \D_{A_0^R}(S)$$
and $$J_{\Phi_v}\D_{A^R}(S) \subseteq \D_{A^R}(S)$$
and as $f$ preserves this filtration by above, $f$ can be uniquely lifted to a map in $\Hom_{\okw}(A_0^S,A^S)$ where $A_0^S$ and $A^S$ are unique lifts of $A_0^R$ and $A^R$ to $S$, respectively. Now using induction on $n$ and using $\cdots \subseteq R/m_R^n \subseteq \cdots \subseteq R/m_R = \fpbar$, we get the proposition.

\end{proof}

Now we consider case 2. Consider the $\okv$-module $L(A_0,A) = \Hom_{\okw}(A_0,A)$ with the Hermitian form defined by $\langle f,g \rangle = \lambda_0^{-1} \circ g^{\vee} \circ \lambda \circ f$ so that for all $x \in \okv$ we have
$$\langle xf,g \rangle = \langle f,\bar{x}g \rangle$$
Now using the above property we can find a unique $K_v$-valued $\okv$-Hermitian form $\langle , \rangle_{\text{CM}}$ on $L(A_0,A)$ satisfying $\langle f,g \rangle = \text{tr}_{K_v/K_{0,v_0}}\langle f,g \rangle_{\text{CM}}$ by a standard argument. 

Let $S = \okv \otimes_{\mathbb{Z}_p} W$, $\text{Fr} \in \text{Aut}\; W$ be the Frobenius automorphism, then on $S$ we have the induced automorphism $(x \otimes w)^{\text{Fr}} = x \otimes w^{\text{Fr}}$. For each $\psi: \okv^u \rightarrow W$, there exists an idempotent $e_{\psi} \in S$ satisfying $(x \otimes 1)e_{\psi} = (1 \otimes \psi(x))e_{\psi}$ for all $x \in \okv^u$. They satisfy $e_\psi^{\text{Fr}} = e_{\text{Fr} \circ \psi}$, $S = \prod_{\psi:{\okv^u \rightarrow W}} e_\psi S$ and $e_\psi S \cong O_{\check{K}_v}$, where $\check{K}_v$ is the maximal unramified extension of $K_v$. Let $m(\psi, \Phi_v) = \# \lbrace \phi \in \Phi_v | \phi|_{\okv^u} = \psi \rbrace$. Let $S_0 = \okw \otimes_{\mathbb{Z}_p} W$, then do the same as above for $S_0$. By Lemma 2.3.1 of \cite{howard}, we have that there exist $b \in S, b_0 \in S_0$ such that
$$L(A_0,A) \cong \lbrace s \in S | (b_0s)^{\text{Fr}} = b^{\text{Fr}}s \rbrace$$

\begin{proposition}
For some $\beta \in F_{v_F}^\times$ satisfying
\[\beta O_K =  \begin{cases} 
      \mathfrak{a}\p_{F_{v_F}} \mathcal{D}_{v_0} \mathcal{D}_v^{-1} O_{K_v} & \text{if}\; K_v/F_{v_F}\;  is \; \text{unramified} \\
      \mathfrak{a}\mathcal{D}_{v_0} \mathcal{D}_v^{-1} O_{K_v} & \text{if}\; K_v/F_{v_F}\; is \; \text{ramified}\; \\ 
   \end{cases}
\]
we have $L(A_0,A) \cong \okv$ as an $O_{K_v}$-module with $\langle x,y  \rangle_{\text{CM}} = \beta x\bar{y}$ on $\okv$.
\end{proposition} 
 \begin{proof}
 In the same way as in lemma 2.3.2 of \cite{howard}, $L(A_0,A)$ is a free $\okv$-module of rank 1, let $s$ be $L(A_0,A) = s\okv$. Again as in lemma 2.3.2 of \cite{howard}, we get $\xi \in S \otimes_{\mathbb{Z}} \mathbb{Q}$ satisfying $\langle a,b \rangle = \xi a \bar{b}$ for $a,b \in L(A_0,A) \subseteq S$ and $\xi S = \mathfrak{a} \mathcal{D}_{v_0} \mathcal{D}_v^{-1} S$. Now we want to compute $s \bar{s} S$. Let $\lbrace \psi^0, \psi^1, \cdots, \psi^{f-1} \rbrace$ be the set of embeddings $\okv^u \rightarrow W$ and $\psi^i_0$ be the restriction of $\psi^{i}$ to $O^u_{K_0,v_0}$ ($0 \leq i \leq f-1$). Now by above we have $(b_0s)^{\text{Fr}} = b^{\text{Fr}}s$, so we get 
 $$\ord_{\psi^{i+1}}(s) = \ord_{\psi^i}(s) - \ord_{\psi^i}(b) + \ord_{\psi^i}(b_0) = \\
\ord_{\psi^i}(s) - m(\psi^i, \Phi_v) + e(K_v/K_{0,v_0})m(\psi_0^i,\Phi_{0,v_0}). $$
Assuming $K_v/F_{v_F}$ is unramified, an easy computation shows
\[m(\psi^i,\Phi_v) - e(K_v/K_{0,v_0})m(\psi_0^i,\Phi_{0,v_0})=  \begin{cases} 
      0 & \text{if}\; \phi^1_1|_{\okv^u} \neq \psi^i,\;\;\; \phi|_{\okw^u} \neq \psi^i_0 \\
      -1 & \text{if}\; \phi^1_1|_{\okv^u} = \psi^i,\;\;\; \phi|_{\okw^u} = \psi^i_0\\
      1 & \text{if}\;  \phi^1_1|_{\okv^u} = \psi^i,\;\;\; \phi|_{\okw^u} \neq \psi^i_0\\
      0 & \text{if}\;  \phi^1_1|_{\okv^u} \neq \psi^i,\;\;\; \phi|_{\okw^u} = \psi^i_0\\
   \end{cases}
\]
 so the sequence $(\ord\psi^0(s), \ord\psi^1(s),\cdots,\ord\psi^{f-1}(s))$ has the form $(0,0,\cdots,0,1,1,\cdots,1,0,\cdots,0)$ with the same number (say $j = \frac{f}{2}$) of 0's and 1's where $\psi^j$ is the restriction of conjugation (nontrivial automorphism of $\text{Gal}(K_v/F_{v_F})$) to $K_v^u$, and we then get 
 $$\ord_{\psi^i}(s) + \ord_{\psi^{i+j}}(s) = 1$$
 for all $i$ and so
 $s\bar{s} = \p_{F_{v_F}}S$.
 
 Now assuming $K_v/F_{v_F}$ ramified, we get $m(\psi^i_0,\Phi_{0,v_0}) = \frac{e(K_{0,v_0}/\qp)}{2}$ and $m(\psi^i, \Phi_v) = e(K_v/K_{0,v_0})m(\psi^i_0,\Phi_{0,v_0})$ so
 $$\ord_{\psi^{i+1}}(s) = \ord_{\psi^i}(s)$$
 for all $i$ and $s\bar{s}S=S$. Let $\epsilon$ be the ramification index of $\tilde{K}_p/K_v$.
 \end{proof}
 
 \begin{proposition}
 Suppose that $f$ is an $O_{K_v}$-module 
 generator of $L(A_0,A)$, then one can lift $f$ 
 to $L^{(k)}\backslash L^{(k+1)}$ with $k = 
 \epsilon \ord_{K_{0,v_0}} \mathcal{D}_{v_0}$ if 
 $K_v/F_{v_F}$ is ramified (resp. $k = \epsilon$ if $K_v/F_{v_F}$ is unramified).
 \end{proposition}
 Let $\D_{v_0},\D_{v}$ be Grothendieck-Messing crystals of $A_0,A$. Now $\ker(\tilde{W}^{(2)} \rightarrow\tilde{W}^{(1)} = \fpbar)$ has a divided power structure compatible with $p\tilde{W}^{2}$ (either the trivial divided power structure if $\tilde{W}/W$ is ramified and the canonical divided powers on $p\tilde{W}^{(2)}$ otherwise), now we have by \cite{howard},
 $$\D_{v_0}(\tilde{W}^{(2)}) \cong S_0 \otimes_W \tilde{W}^{(2)}$$
 $$\D_v(\tilde{W}^{(2)}) \cong S \otimes_W \tilde{W}^{(2)}$$
 Hodge filtrations are $J_{\Phi_{0,v_0}}\D_{v_0}(\tilde{W}^{(2)})$ and $J_{\Phi_v}\D_v(\tilde{W}^{(2)})$ and $f$ lifts to a map $A_0^{(2)} \rightarrow A^{(2)}$ (where $A_0^{(2)}$ and $A^{(2)}$ are unique deformations of $A_0$ and $A$ to $\tilde{W}^{(2)}$) iff 
 $$f: J_{\Phi_{0,v_0}}\D_{v_0}(\tilde{W}^{(2)}) \rightarrow \D_v(\tilde{W}^{(2)})/J_{\Phi_v}\D_v(\tilde{W}^{(2)})$$
 is trivial. If $f \in \Hom_{\okw}(A_0,A) \subseteq S$ corresponds to $s \in S$, then we consider the multiplication by $s$ 
 $$J_{\Phi_{0,v_0}}(S_0 \otimes_W \tilde{W}) \rightarrow (S \otimes_W \tilde{W})/J_{\Phi_v}(S \otimes_W \tilde{W}).$$
 Now by mapping 
 $$(S \otimes_W \tilde{W})/J_{\Phi_v}(S \otimes_W \tilde{W}) \xrightarrow{(\bar{\phi}^1_1,\phi^2_1,\cdots,\phi^n_1)} \cp^{nd}$$
 Firstly, assuming $K_v/F_{v_F}$ is unramified, $\phi_i^j(s) = 0$ for all $(i,j) \neq (1,1)$ and for $\bar{\phi}^1_1$, it goes to
 $$\bar{\phi}^1_1(s)\prod_{\phi|_{K_{0,v_0}^u} = \phi_1|_{K_{0,v_0}^u}} (\phi_1(s) - \phi(s))$$
 where the product is over the $\phi: K_{0,v_0} \rightarrow \cp$ with the aforementioned property. Now as $\bar{\phi}_1|_{\okw^u} \neq \phi_1|_{\okw^u}$ all components of product above are units except for $\bar{\phi}^1_1(s)$ which has valuation 1 in $W$, so valuation $\epsilon$ in $\tilde{W}$. So using a similar idea for $\tilde{W}^{(1)}, \tilde{W}^{(2)}, \cdots, \tilde{W}^{(\epsilon)}$ we can lift $f$ to them, but not to $\tilde{W}^{(\epsilon+1)}$.
 
 Secondly, if $K_v/F_{v_F}$ is ramified, suppose 
 that we extended $f$ to $\tilde{W}^{(k)}$, then 
 by prop. 2.3.3 of \cite{howard}, $s \in 
 S^{\times}$ so the induced map on Dieudonne 
 modules $D(A_0) \otimes_{O_{F_{0_{v_{0,F_0}}}}} 
 O_{F_{v_F}} \rightarrow D(A)$ is an isomorphism. So $f$ induces an isomorphism of Lie algebras 
 $$\text{Lie}(A_0) \otimes_{O_{F_{0_{v_{0,F_0}}}}} O_{F_{v_F}} \cong \text{Lie}(A)$$
 Now as we assumed $f$ extends to $\tilde{W}^{(k)}$, Nakayama's lemma implies that the induced map 
 $$\text{Lie}(A_0^{(k)}) \otimes_{O_{F_{0_{v_{0,F_0}}}}} O_{F_{v_F}} \cong \text{Lie}(A^{(k)})$$
 where $A_0^{(k)}$ and $A^{k}$ are deformations of $A_0$ and $A$ to $\tilde{W}^{(k)}$. So in $\tilde{W}^{(k)}[t]$, we have 
 $$\prod_{\phi \in \Phi_v}(t-\phi(x)) = \prod_{\phi \in \tilde{\Phi}_{0,v_0}}(t-\phi(x))$$ 
 where $\tilde{\Phi}_{0,v_0}$ is 
 $\lbrace \overline{\phi^1_1}, \phi^2_1, \cdots, \phi^n_1, \phi_2^1, \cdots \rbrace$. So we get that $\phi^1_1 = \overline{\phi^1_1} (\text{mod}\; m^k)$ which implies that $k \leq \epsilon \ord_{K_{0,v_0}} \mathcal{D}_{v_0}$.\\
 Now suppose that $k \leq \epsilon \ord_{K_{0,v_0}} \mathcal{D}_{v_0}$, then the $\okv$-action on $A_0^{(k)} \otimes_{O_{F_{0_{v_{0,F_0}}}}} O_{F_{v_F}}$ satisfies $\Phi_v$-determinant condition and so $f: A_0 \otimes_{O_{F_{0_{v_{0,F_0}}}}} O_{F_{v_F}} \rightarrow A$ is an isomorphism of $p$-divisible groups and so one can see $A_0^{(k)} \otimes_{O_{F_{0_{v_{0,F_0}}}}} O_{F_{v_F}}$ as a deformation of $A$ to $\tilde{W}^{(k)}$. By uniqueness of such deformations, there exists an $\okv$-linear isomorphism 
 $$A_0^{(k)} \otimes_{O_{F_{0_{v_{0,F_0}}}}} O_{F_{v_F}} \rightarrow A^{(k)}$$ lifting $f$, so by composing with $A_0^{(k)} \hookrightarrow A_0^{(k)} \otimes_{O_{F_0,v_{0,F_0}}} O_{F_{v_F}}$, we get the lift $A_{0}^{(k)} \rightarrow A^{(k)}$ of $f$.
 
 \begin{proposition}
 Let $\pi_{K_v}$ be a uniformizer of $\okv$. If $f \in L^{(k)}$, then $\pi_{K_v}f \in L^{(k+\epsilon)}$ and the multiplication map by $\pi_{K_v}$ map induces an injective map $L^{(k)}\backslash L^{(k+1)} \rightarrow L^{(k+\epsilon)}\backslash L^{(k+\epsilon+1)}$.
 \end{proposition}
\begin{proof}
Let $\D_0^{(k)}$, $\D^{(k)}$ be Grothendieck-Messing crystals of $A_0^{(k)}$ and $A^{(k)}$, now $\tilde{W}^{(k+\epsilon)} \rightarrow \tilde{W}^{(k)}$ is a PD-thickening
\[ \begin{tikzcd}
J_{0,\Phi_{0,v_0}}\D_0^{(k)}(\tilde{W}^{(k+\epsilon)}) \arrow{r}{f^{(k)}} \arrow[swap]{d}{} & \D^{(k)}(\tilde{W}^{(k+\epsilon)})/J_{\Phi_v}(\D^{(k)}(\tilde{W}^{k+\epsilon})) \arrow{d}{} \\%
J_{0,\Phi_{0,v_0}}\D_0^{(k)}(\tilde{W}^{(k)}) \arrow{r}{f^{(k)}}& \D^{(k)}(\tilde{W}^{(k)})/J_{\Phi_v}(\D^{(k)}(\tilde{W}^{(k)}))
\end{tikzcd}
\]
bottom row is the zero map, so the top row becomes zero after $\otimes_{\tilde{W}^{(k+\epsilon)}} \tilde{W}^{(k)}$, so its image annihilated by $m^{\epsilon}$, and so
$\pi_{K_v} f^{(k)} = \phi^1_1(\pi_{K_v}) f^{(k)}$ is zero on top row and so can be lifted to $L^{(k+\epsilon)}$. Now suppose that $f \in L^{(k)}$, we have the PD-thickening 
$\tilde{W}^{(k+\epsilon+1)} \rightarrow \tilde{W}^{(k)}$, so we get a diagram
\[ \begin{tikzcd}
J_{0,\Phi_{0,v_0}}\D_0^{(k)}(\tilde{W}^{(k+\epsilon+1)}) \arrow{r}{f^{(k)}} \arrow[swap]{d}{} & \D^{(k)}(\tilde{W}^{(k+\epsilon+1)})/J_{\Phi_v}(\D^{(k)}(\tilde{W}^{(k+\epsilon+1)})) \arrow{d}{} \\%
J_{0,\Phi_{0,v_0}}\D_0^{(k)}(\tilde{W}^{(k+\epsilon)}) \arrow{r}{f^{(k)}} \arrow[swap]{d}{}& \D^{(k)}(\tilde{W}^{(k+\epsilon)})/J_{\Phi_v}(\D^{(k)}(\tilde{W}^{(k+\epsilon)})) \arrow{d}{}\\%
J_{0,\Phi_{0,v_0}}\D_0^{(k)}(\tilde{W}^{(k+1)}) \arrow{r}{f^{(k)}} \arrow[swap]{d}{}& \D^{(k)}(\tilde{W}^{(k+1)})/J_{\Phi_v}(\D^{(k)}(\tilde{W}^{(k+1)})) \arrow{d}{} \\%
J_{0,\Phi_{0,v_0}}\D_0^{(k)}(\tilde{W}^{(k)}) \arrow{r}{f^{(k)}}& \D^{(k)}(\tilde{W}^{(k)})/J_{\Phi_v}(\D^{(k)}(\tilde{W}^{(k)}))
\end{tikzcd}
\]
now assume that $\pi_{K_v} f^{(k)}$ can be lifted to $L^{(k+\epsilon+1)}$, so in the top row of the diagram above $f^{(k)}$ has image inside $m^{k+1}$, so the map $J_{0,\Phi_{v_0}}\D_0^{(k)}(\tilde{W}^{(k+1)}) \rightarrow \D^{(k)}(\tilde{W}^{(k+1)})/J_{\Phi_v}\D^{(k)}(\tilde{W}^{(k+1)})$ gotten by the PD-thickening 
$\tilde{W}^{(k+1)} \rightarrow \tilde{W}^{(k)}$ is zero and $f^{(k)}$ can be lifted to $L^{(k+1)}$.
\end{proof}

\begin{theorem}
Suppose that the ramification condition is satisfied, then for any nonzero $f \in L(A_0,A)$ with $\langle f,f \rangle = \alpha$,  we have $f \in L^{(k)}\backslash L^{(k+1)}$ where 
$$k=\frac{1}{2}\ord_{\p}(\alpha \mathfrak{a}^{-1}\p_{F_{v_F}} \mathcal{D}_{v_0}^{-1} \mathcal{D}_v).$$
\end{theorem}
\begin{proof}
Let $f = f_0 \pi_{K_v}^n$ for an $O_{K_v}$-module generator $f_0$ of $L(A_0,A)$. By previous proposition, we know that we can lift $f$ to $L^{(n+1)\epsilon}\backslash L^{(n+1)\epsilon+1}$. In order to compute $k = (n+1)\epsilon$ in terms of $\alpha$, we have

If $K_v/F_{v_F}$ is unramified, 
$$\alpha \okv = \langle f,f \rangle\okv = \p_{K_v}^{2n+1}\mathfrak{a} \mathcal{D}_{v_0} \mathcal{D}_v^{-1}\okv$$
so $$n = \frac{\ord_{K_v}(\alpha \mathfrak{a}^{-1}\mathcal{D}_{v_0}^{-1} \mathcal{D}_v) -1}{2} \Rightarrow (n+1)\epsilon = \frac{\ord_{\tilde{K}_p}(\alpha\mathfrak{a}^{-1}\mathcal{D}_{v_0}^{-1} \mathcal{D}_v \p_{F_{v_F}})}{2}$$
If $K_v/F_{v_F}$ is ramified, $\alpha \okv = \langle f,f \rangle \okv = \p_{K_v}^{2n} \mathfrak{a} \mathcal{D}_{v_0} \mathcal{D}_v^{-1} \okv$
so
$$n = \frac{1}{2}\ord_{K_v}(\alpha \mathfrak{a}^{-1}\mathcal{D}_{v_0} \mathcal{D}_v^{-1}) \Rightarrow (n+1)\epsilon = \frac{\ord_{\tilde{K}_p}(\alpha\mathfrak{a}^{-1}\mathcal{D}_{v_0}^{-1} \mathcal{D}_v \p_{F_{v_F}})}{2}$$
\end{proof}

\section{Global part}
\subsection{Rapoport-Smithling-Zhang Shimura varieties} 
Here we shall recall some notions from \cite{rsz} that we are going to use later. We use notations from notations section freely. Also if $\p$ is a prime in $\tilde{K}$, we assume that $\tilde{K}_p$ in the notations section is $\tilde{K}_{\p}$.
We first define the Deligne-Mumford stack $\mathcal{M}_0$ over $O_{\tilde{K}}$. 
This is the Deligne-Mumford stack that for a scheme $S$ over $O_{\tilde{K}}$, gives the groupoid of tuples $(A_0,\iota_0,\lambda_0)$ with $A_0$ an abelian 
scheme over $S$ with $O_{K_0}$-action $\iota: O_{K_0} \rightarrow \text{End} A_0$ with $\Phi_0$-Kottwitz condition:
$$\text{charpol}(\iota_0(a))|_{\text{Lie}(A_0)}(t) = \prod_{\phi \in \Phi_0} (t-\phi(a))$$
for all $a \in O_{K_0}$ and $\lambda: A_0 \rightarrow A_0^{\vee}$ is a principal polarization such that its Rosati involution on $O_{K_0}$ by $K_0$ is the nontrivial conjugation of $K_0/F_0$.  Now consider $L_{\Phi_0}$ to be the set of isomorphism classes of pairs $(\Lambda_0, \langle , \rangle_0)$ where $\Lambda_0$ is a locally free $O_{K_0}$-module of rank 1 with a nondegenerate alternating pairs $\langle , \rangle_0: \Lambda_0 \times \Lambda_0 \rightarrow \mathbb{Z}$ with $\langle ax, y \rangle =  \langle x, \bar{a}y \rangle$ for all $x,y \in O_{K_0}$ such that $x \rightarrow \langle \sqrt{\Delta}x, x\rangle_0$ is negative definite quadratic form on $\Lambda_0$ and that $\Lambda_0$ inside $\Lambda_0 \otimes \mathbb{Q}$ is a self-dual lattice. $L_{\Phi_0}$ is finite (page 11 of \cite{rsz}) and an object $(A_0, \iota_0, \lambda_0) \in \mathcal{M}_0(\mathbb{C})$ gives a unique element of $L_{\Phi_0}$ by considering $H_1(A_0(\mathbb{C}), \mathbb{Z})$ endowed with Riemann form, which gives a bijection between isomorphism classes of objects of $\mathcal{M}_0(\mathbb{C})$ and $L_{\Phi_0}$.

Now the point of the previous paragraph is that this decomposition extends to the whole integral model $\mathcal{M}_0$ over $O_{\tilde{K}}$ using the following equivalence relation  $L_{\Phi_0}: \Lambda_0 \cong \Lambda_0^{\prime}$
if $\Lambda_0 \otimes \hat{\mathbb{Z}}$ and $\Lambda_0^{\prime} \otimes \hat{\mathbb{Z}}$ are $\hat{O}_{K_0}$-linearly similar up to a factor in $\hat{\mathbb{Z}}^{\times}$ and if $\Lambda_0 \otimes \mathbb{Q}$ and $\Lambda_0^{\prime} \otimes \mathbb{Q}$ are $K_0$-linearly similar up to a factor in $\mathbb{Q}^{\times}$. Now the decomposition is as follows:

\begin{proposition}
(Lemma 3.4 in \cite{rsz}) The stack $\mathcal{M}_0$ admits a decomposition to open and closed substacks $\mathcal{M}_0 = \coprod_{\xi \in L_{\Phi_0}/\cong} \mathcal{M}_0^{\xi}$ (on the level of $\mathbb{C}$-points, this is just the equivalence class of $H_1(A_0(\mathbb{C}), \mathbb{Z})$ as explained above).
\end{proposition}

We are now ready to introduce the ambient Deligne-Mumford stack $\mathcal{M}$ that will be the stack in which we are going to have special cycles on. Fix a free $O_{K_0}$-module $W$ of rank $d$ equipped with a nondegenerate $K_0/F_0$ Hermitian form. For a prime $v_0$ of $K_0$, let $W_{v_0}$ be the completion of $W$ at $v_0$. Let $\mathcal{M}$ over $O_{\tilde{K}}$ be the Deligne-Mumford stack that for each $O_{\tilde{K}}$-scheme $S$ gives the groupoid of tuples $(A_0,\iota_0,\lambda_0,A,\iota,\lambda)$ where $(A_0,\iota_0,\lambda_0) \in \mathcal{M}_0^{\xi}(S)$ for some $\xi \in L_{\Phi_0}/ \cong$ and $A/S$ is an abelian scheme with $O_{K_0}$-action $\iota: O_{K_0} \rightarrow \text{End} A$ with Kottwitz condition:
$$\text{charpol}(\iota(a)|_{\text{Lie}(A)})(t) = (t-\phi_1(a))^{n-1} (t-\bar{\phi}_1(a)) \prod_{\phi \in \Phi_0\backslash \lbrace \phi_1 \rbrace} (t-\phi(a))^n$$
and $\lambda$ is a principal polarization whose Rosati involution on $O_{K_0}$ by $\iota$ gives the nontrivial conjugation $K_0/F_0$. Also  impose the sign condition:
$$inv_v^r(A_{0,s}, \iota_{0,s}, \lambda_{0,s}, A_s, \iota_s, \lambda_s) = inv_v(-W_v)$$
(See appendix $A$ of \cite{rsz} for the definition of $inv_v^r$) for any $s \in S$ and $v$ a finite place of $F_0$ nonsplit in $K_0$, also we assume that for any place $\p$ of $\tilde{K}$ with $\p$ its residue characteristic, the triple $(A \otimes \mathbb{Z}_{(p)}, \iota \otimes \mathbb{Z}_{(p)}, \lambda \otimes \mathbb{Z}_{(p)})$ over $S \times_{spec O_{\tilde{K}}} O_{\tilde{K}, (\p)}$ satisfies the conditions in section 4 of \cite{rsz}. One of the results in \cite{rsz} is the following:

\begin{theorem}
(Theorem 5.2 of \cite{rsz}) $\mathcal{M}$ over $O_{\tilde{K}}$ is representable by a Deligne-Mumford stack. $\mathcal{M}$ is flat over $O_{\tilde{K}}$ and smooth of relative dimension $d-1$ over $O_{\tilde{K}}$ after removing all $\p \in spec O_{\tilde{K}}$ with AT-type (1) or (4) (refer to section 4.4 of \cite{rsz} for the definition of AT-type).
\end{theorem}

\subsection{Stacks $\mathcal{Z}(\alpha)$ and $\mathcal{X}$}
First we define the Deligne-Mumford stack $\cm_{\Phi}^{\mathfrak{a}}$:
\begin{definition}
Let $\cm_{\Phi}^{\mathfrak{a}}$ be the Deligne-Mumford stack over $O_{\tilde{K}}$ such that for each $S$ over $O_{\tilde{K}}$, we get $\cm_{\Phi}^{\mathfrak{a}}(S)$ is the groupoid of $(A,\iota, \lambda)$ with:
\begin{itemize}
\item $A/S$ is an abelian scheme of relative dimension $dn$.
\item $\iota: O_K \rightarrow \text{End} A$ satisfies $\Phi$-Kottwitz condition:
$$\text{charpol}(\iota(a)_{\text{Lie}(A)})(t) = \prod_{\phi \in \Phi}(t-\phi(a))\;\;\;\;\;\;\;\;\; \forall a \in O_K$$ 
\item $\lambda: A \rightarrow A^{\vee}$ is a polarization with kernel $A[\mathfrak{a}]$ whose Rosati involution on $O_K \subseteq \text{End} A$ gives the nontrivial involution of $K/F$.
\end{itemize}
\end{definition}
Fix $\xi \in L_{\Phi_0}/\cong$ from now on. Now we define the algebraic stack $\mathcal{X}$ to be the substack of $\mathcal{M}_0^{\xi} \times_{O_{\tilde{K}}} \cm^{\mathfrak{a}}_{\Phi}$ whose $S$-points (for an $O_{\tilde{K}}$-scheme $S$) consists 
of $(A_0,\iota_0,\lambda_0, A,\iota, \lambda)$ with $inv_v^r(A_{0,s}, \iota_{0,s}, \lambda_{0,s}, A_s, \iota_s, \lambda_s) = inv_v(-W_v)$ and $(A \otimes \mathbb{Z}_{(p)}, \iota \otimes 
\mathbb{Z}_{(p)}, \lambda \otimes \mathbb{Z}_{(p)})$ satisfies the conditions of section 4 of \cite{rsz}, then we have a forgetful map $\mathcal{X} \rightarrow \mathcal{M}$ by sending $(A_0,\iota_0, \lambda_0,A,\iota,\lambda)$ in $\mathcal{X}(S)$ to $(A_0,\iota_0,\lambda_0,A,\iota|_{O_{K_0}}, \lambda)$ in $\mathcal{M}(S)$. Now it follows from \cite{howard} prop. 3.1.2 that $\mathcal{X} \rightarrow \mathcal{M}$ is étale and proper. By the same proposition, for all $(A,\iota, \lambda) \in \cm^{\mathfrak{a}}_{\Phi}(\ktildebarp)$ we have a unique canonical lift $(A^{can}, \iota^{can}, \lambda^{can})$ to $\cm^{\mathfrak{a}}_{\Phi}(\tilde{W})$. Also for all $(A_0,\iota_0,\lambda_0) \in \mathcal{M}_0^{\xi}(\ktildebarp)$, we have a unique canonical lift $(A_0^{can}, \iota_0^{can}, \lambda_0^{can}) \in \mathcal{M}_0^{\xi}(\tilde{W})$

\begin{proposition}
We have $\mathcal{D}_0\mathcal{D}^{-1} = \partial_{F/F_0}^{-1}O_K$.
\end{proposition}
\begin{proof}
An easy calculation using the ramification condition introduced in notations section.
\end{proof}

Now fix a sextuple $(A_0,\iota_0,\lambda_0,A,\iota,\lambda) \in \mathcal{X}(S)$ for some $O_{\tilde{K}}$-scheme $S$. We can define a Hermitian space as follows: Consider $\Hom_{O_{K_0}}(A_0,A)$, this has a normal $O_{K_0}$-valued Hermitian form given by:
$$\langle f,g \rangle = \lambda_0^{-1} \circ g^{\vee} \circ \lambda \circ f$$
As $A$ has $O_K$-action, we can change the Hermitian form and define $\langle , \rangle_{CM}$ to be the unique $K$-valued Hermitian form satisfying 
$$\langle f,g \rangle = \text{tr}_{K/K_0} \langle f,g \rangle_{CM}$$
\begin{proposition}
Suppose $k$ is an algebraically closed field and that $(A_0,\iota_0, \lambda_0, A, \iota, \lambda) \in \mathcal{X}(k)$, if there is $f \in Hom_{O_{K_0}}(A_0,A) \otimes \mathbb{Q}$ with $\langle f,f \rangle_{CM} \neq 0$, then char $k \neq 0$ and $A_0 \otimes_{O_{F_0}} O_F$ and $A$ are $O_K$-isogenous.
\end{proposition}
\begin{proof}
$f: A_0 \rightarrow A$ induces the $O_{K_0}$-linear map $\tilde{f}: A_0 \otimes_{O_{F_0}} O_F \rightarrow A$ (where $A_0 \otimes_{O_{F_0}} O_F$ is the abelian variety over $K$ defined by Serre construction and having action by $O_{K_0} \otimes_{O_{F_0}} O_F = O_K$). Now for $l \nmid char\; k$, let $T_l(A)$, $T_l(A_0)$ be Tate modules of $A$ and $A_0$ respectively and $T_l^0(A_0) = T_l(A_0) \otimes_{\mathbb{Z}_l} \mathbb{Q}_l$ and $T^0_l(A) = T_l(A) \otimes_{\mathbb{Z}_l} \mathbb{Q}_l$. The polarization $\lambda_0$ gives $\mathbb{Q}_l$-linear map $T_l^0(A_0) \times T^0_l(A_0) \rightarrow \mathbb{Q}_l(1) \rightarrow F_{0,l}(1)$ so that by tensoring $\otimes_{F_{0,l}} F_l$ gives 
$$\Lambda_0: T^0_l(A_0 \otimes_{O_{F_0}} O_F) \times T^0_l(A_0 \otimes_{O_{F_0}} O_F) \rightarrow F_l(1)$$
Also polarization $\lambda$ gives $\mathbb{Q}_l$-linear map
$$\Lambda: T^0_l(A) \times T^0_l(A) \rightarrow \mathbb{Q}_l(1)$$
so that $\Lambda$ can be written uniquely as $\text{tr}_{F_l/\mathbb{Q}_l} \tilde{\Lambda}$ for some 
$$\tilde{\Lambda}: T^0_l(A) \times T^0_l(A) \rightarrow F_l(1)$$
Now $\tilde{f}$ gives us a $\mathbb{Q}_l$-linear map $\tilde{f}_l: T^0_l(A_0 \otimes_{O_{F_0}} O_F) \rightarrow T^0_l(A)$ and we call the adjoint of $\tilde{f}_l$ by $\tilde{f}^{\dag}_l$ (which is the unique $\mathbb{Q}_l$-linear map $\tilde{f}^{\dag}_l: T_l^0(A) \rightarrow T_l^0(A_0 \otimes_{O_{F_0}} O_F)$ for which $\Lambda_0(x,\tilde{f}^{\dag}_l(y)) = \tilde{\Lambda}(\tilde{f}_l(x),y)$) for all $x \in T_l^{0}(A_0 \otimes_{O_{F_0}} O_F)$ and $y \in T_l^0(A)$. Now we have $\langle f,f \rangle_{CM} = \tilde{f}^{\dag}_l \circ \tilde{f}_l$ as elements of $F_l \subseteq \text{End}_{\mathbb{Q}_l}(T^0_l(A_0 \otimes_{O_{F_0}} O_F))$, so that by $\langle f,f \rangle_{CM} \neq 0$, we have that $\tilde{f}_l$ is injective, so $\tilde{f}: A_0 \otimes_{O_{F_0}} O_F \rightarrow A$ is an $O_K$-isogeny, so $A$ is isogenous to $A_0 \otimes_{O_{F_0}} O_F \cong A_0 \times A_0 \times \cdots A_0$ (d times). This isogeny cannot happen if $char\; k = 0 $ as the signatures of $A$ and  $A_0 \otimes_{O_{F_0}} O_F$ are different. 
\end{proof}

Now we have the theorem relating the local parts of $L(A_0,A) = \Hom_{O_{K_0}}(A_0,A)$ to $\Hom_{O_{K_0}}(A_0[\mathfrak{q}_0^{\infty}],A[\mathfrak{q}^\infty])$:

\begin{proposition}
Suppose that $k$ is an algebraically closed field of $char\; k > 0$. Suppose that $(A_0,\iota_0,\lambda_0,A,\iota,\lambda) \in \mathcal{X}(k)$ is such that $A_0 \otimes_{O_{F_0}} O_F$ and $A$ are $O_k$-isogenous, then $L(A_0,A)$ is a projective $O_k$-module of rank 1 and letting $\mathfrak{q}$ be a prime of $F$ over the prime $\mathfrak{q}_0$ of $F_0$ over the rational prime $q$, the map
$$L(A_0,A) \otimes_{O_F} O_{F,\mathfrak{q}} \rightarrow \Hom_{O_{K_0}}(A_0[\mathfrak{q}_0^\infty],A[\mathfrak{q}^\infty])$$
is an isomorphism.
\end{proposition}

\begin{proof}
We have an $O_K$-isogeny $A \rightarrow A_0 \otimes_{O_{F_0}} O_F$, so this induces a map
$$\Hom_{O_{K_0}}(A_0,A) \rightarrow \Hom_{O_{K_0}}(A_0,A_0 \otimes_{O_{F_0}} O_F) = \Hom_{O_{K_0}}(A_0,A_0^d) \cong O_{K_0}^d$$
which is an injection with finite cokernel, so that $\Hom_{O_{K_0}}(A_0,A)$ is a projective $O_{K_0}$-module of rank $d$. Also for $p$-divisible group, we have 
$$\Hom_{O_{K_0}}(A_0[\mathfrak{q}_0^{\infty}],A[\mathfrak{q}_0^\infty]) \rightarrow \Hom_{O_{K_0}}(A_0[\mathfrak{q}_0^{\infty}],(A_0 \otimes_{O_{F_0}} O_F)[\mathfrak{q}_0^{\infty}]) = \Hom_{O_{K_0}}(A_0[\mathfrak{q}_0^{\infty}],A_0[\mathfrak{q}_0^{\infty}]^d) = O_{K_0,\mathfrak{q}_0}^d$$
is injective with finite kernel, so $\Hom_{O_{K_0}}(A_0[\mathfrak{q}_0^{\infty}],A[\mathfrak{q}_0^{\infty}])$ is a projective $O_{K_0,\mathfrak{q}_0}$-module of rank $d$. Also the map
$$\Hom_{O_{K_0}}(A_0,A) \otimes_{O_{F_0}} O_{F_0,\mathfrak{q}_0} \rightarrow \Hom_{O_{K_0}}(A_0[\mathfrak{q}_0^{\infty}],A[\mathfrak{q}_0^{\infty}])$$
is injective with $\mathbb{Z}_q$-torsion free cokernel and also the $\mathbb{Z}_q$-rank of domain and codomain are equal by above, so it is an isomorphism. Also $L(A_0,A)$ is a projective $O_K$-module a fortiori of rank 1 by being a projective $O_{K_0}$-module of rank $d$. Taking $\mathfrak{q}$-parts from both sides we get the statement of proposition.
\end{proof}

We define the groups $C_K$ and $C^0_{K}$ in the same way as in \cite{cheraghi}. For an ideal $I$ of $O_K$, let $A^I= I \otimes_{O_K} A$ be the abelian variety constructed by the Serre construction. 

\begin{proposition}
Suppose that $S$ is a connected $O_{\tilde{K}}$-scheme and $(A_0,\iota_0,\lambda_0,A,\iota, \lambda) \in \mathcal{X}(S)$. For each $(I,\zeta) \in C_K$, we have an $\hat{O}_K$-linear isomorphism
$$\hat{L}(A_0,A^I) \cong \hat{L}(A_0,A)$$
where the Hermitian form $\langle , \rangle_{CM}^{I}$ on left is $gen(I)\langle , \rangle_{CM}$ on the right (the map $gen$ is defined in \cite{cheraghi}).
\end{proposition}
\begin{proof}
Same as prop 3.3.1 of \cite{howard}.
\end{proof}
For a pair of abelian varieties $(A_0,A) \in \mathcal{X}(\mathbb{C})$, define 
$$L_{Betti}(A_0,A) = \Hom_{O_{K_0}}(H_1(A_0,\mathbb{C}),H_1(A,\mathbb{C}))$$
Now we have the following structure theorem for $L_{Betti}$:
\begin{proposition}
There is $\beta \in \hat{F}^{\times}$ with $\beta O_F = \partial_{F/F_0}^{-1}\mathfrak{a}$ with an isomorphism 
$$(\hat{L}_{Betti}(A_0,A), \langle , \rangle_{CM}) \cong (\hat{O}_K, \beta x \bar{y})$$
Also showing $\langle x,y \rangle_{CM}$ at archimedean places by $\beta x \bar{y}$ as well, we get that $\beta$ is negative definite at $\infty^{sp} = \phi_1^1|_F$ and positive definite at other archimedean places of $F$.
\end{proposition}
\begin{theorem}
Let $\p$ be a prime of $\tilde{K}$ with $\p_F$ nonsplit in $K$, and let 
$$(A_0, A) \in \mathcal{X}(\ktildebarp)$$
then there is an isomorphism $(\hat{L}(A_0,A), \langle , \rangle_{CM}) \cong (\hat{O}_K, \beta x \bar{y})$ with $\beta \in \hat{F}^{\times}$ such that $\beta O_F = \mathfrak{a} \partial_{F/F_0}^{-1} \p_F^{\epsilon_{\p_F}}$. Also we have $\chi(\beta^{\infty}) = 1$ ($\beta^{\infty}$ is the element of $\mathbb{A}_F^{\times}$ that has trivial archimedean components and at finite places, it is the same as $\beta \in \hat{F}^{\times}$).
\end{theorem}

Let $A_0^{\prime}$ (resp. $A^{\prime}$) be the unique lift of $A_0$ (resp. $A$) to $\mathbb{C}_{p}$ and by fixing $\tilde{K}$-linear isomorphism $\mathbb{C} \cong \mathbb{C}_{p}$, we see $A_0^{\prime}$ and $A^{\prime}$ as abelian varieties in $\mathbb{C}$. Now for a prime $\mathfrak{q}$ of $F$ with $\mathfrak{q}_0$ below it in $F_0$, there are isomorphisms of $O_{K,\mathfrak{q}}$-Hermitian spaces
$$L_{Betti}(A_0^{\prime},A^{\prime}) \otimes_{O_K} O_{K,\mathfrak{q}} \cong \Hom_{O_{K_0}}(A_0^{\prime}[\mathfrak{q}_0^{\infty}], A^{\prime}[\mathfrak{q}^{\infty}])$$
because of the fact that $A_0^{\prime}[\mathfrak{q}_0^{\infty}]$ and $A^{\prime}[\mathfrak{q}^{\infty}]$ are constant $p$-divisible groups. Also by proposition 4.6 there is an isomorphism 
$$\Hom_{O_{K_0}}(A_0,A) \otimes_{O_K} O_{K,\mathfrak{q}} \cong \Hom_{O_{K_0}}(A_0[\mathfrak{q}_0^{\infty}], A[\mathfrak{q}^\infty])$$
Now we have the following lemma:
\begin{lemma}
If $\mathfrak{q}$ is not $\p_F$, then there's an $O_K$-linear isomorphism of Hermitian spaces 
$$Hom_{O_{K_0}}(A_0^{\prime}[\mathfrak{q}_0^{\infty}], A^{\prime}[\mathfrak{q}^{\infty}]) \cong Hom_{O_{K_0}}(A_0[\mathfrak{q}_0^{\infty}], A[\mathfrak{q}^{\infty}])$$
\end{lemma}
\begin{proof}
Recall that $p$ is the characteristic of $\tilde{k}_{\p}$ and $A_0$ and $A$ are defined over $\tilde{k}_{\p}$. If the rational prime below $\mathfrak{q}$ is not $p$ (and call it $q$), then the Tate modules of $A_0$ and $A_0^{\prime}$, and also the Tate modules of $A$ and $A^{\prime}$ are going to be canonically isomorphic, so
$$\Hom_{O_{K_0}}(A_0^{\prime}[\mathfrak{q}_0^{\infty}], A^{\prime}[\mathfrak{q}^{\infty}]) \cong \Hom_{O_{K_0}}(A_0[\mathfrak{q}_0^{\infty}], A[\mathfrak{q}^{\infty}])$$
taking $\mathfrak{q}$-parts we get the wanted isomorphism. It is clear that it respects the Hermitian forms. Now suppose that the rational prime below $\mathfrak{q}$ is $p$, now because $\mathfrak{q} \neq \p_F$ by hypothesis, we have that the set $\Phi({\mathfrak{q}})$ of all embeddings $\phi: K \rightarrow \mathbb{C}_{p}$ with $\mathfrak{q} = \phi^{-1}(pO_{\mathbb{C}_p})$ satisfies $\phi^1_1 \notin \Phi(\mathfrak{q})$ (because of the fact that $\p_F$ is the prime below $\p$ using the inclusion $\bar{\phi}^1_1(F) \subseteq \bar{\phi}^1_1(K) \subseteq \tilde{K}$). Similarly let $\Phi_0(\mathfrak{q}_0)$ be all embeddings $\phi: K_0 \rightarrow \mathbb{C}_p$ with $\mathfrak{q}_0 = \phi^{-1}(pO_{\mathbb{C}_p})$. Then we have the following relation between $\Phi(q)$ and $\Phi_0(\mathfrak{q}_0)$:
$$\Phi(\mathfrak{q}) = \lbrace \phi: K \rightarrow \mathbb{C}_p| \phi|_{K_0} \in \Phi_0(\mathfrak{q}_0) \rbrace$$
Also that $A[\mathfrak{q}^\infty]$ (resp. $A_0[\mathfrak{q}_0^{\infty}]$) is a CM $p$-divisible group with action $O_{K,\mathfrak{q}}$ (resp. $O_{K_0,\mathfrak{q}_0}$) and $\Phi(\mathfrak{q})$ (resp. $\Phi_0(\mathfrak{q}_0)$)-determinant condition. Letting $A_0^{can}$ and $A^{can}$ be the unique lifts of $A_0$ and $A$ respectively to $\tilde{W}$. We have that 
$$\Hom_{O_{K_0}}(A_0^{can}[\mathfrak{q}_0^{\infty}], A^{can}[\mathfrak{q}^\infty]) \rightarrow \Hom_{O_{K_0}}(A_0[\mathfrak{q}_0^{\infty}], A[\mathfrak{q}^{\infty}])$$
is an isomorphism. Now base change $\tilde{W} \cong \mathbb{C}_p$ defines an injection
$$F: \Hom_{O_{K_0}}(A^{can}_0[\mathfrak{q}_0^{\infty}],A^{can}[\mathfrak{q}^{\infty}]) \rightarrow \Hom_{O_{K_0}}(A_0^{\prime}[\mathfrak{q}_0^{\infty}], A^{\prime}[\mathfrak{q}^{\infty}])$$
Now we have Tate's theorem which says for two $p$-divisible groups $G, H$ with Tate modules $TG, TH$ respectively (over specific types of rings $R$ including $\tilde{W}$ and $\cp$ with $E = Frac(R)$) the map
$$\Hom(G,H) \rightarrow \Hom_{Gal(\bar{E}/E)}(TG, TH)$$
is an isomorphism. So the image of $F$ is $Aut(\cp/Frac(\tilde{W}))$-invariants of $\Hom_{O_{K_0}}(A_0^{\prime}[\mathfrak{q}_0^{\infty}], A^{\prime}[\mathfrak{q}^{\infty}])$ so that the map has $\mathbb{Z}_p$-torsion-free cokernel. Now propositions 4.6 and 4.8 and isomorphisms
$$L_{Betti}(A_0^{\prime},A^{\prime}) \otimes_{O_K} O_{K,\mathfrak{q}} \cong \Hom_{O_{K_0}}(A_0^{\prime}[\mathfrak{q}_0^{\infty}], A^{\prime}[\mathfrak{q}^{\infty}])$$
$$\Hom_{O_{K_0}}(A_0,A) \otimes_{O_K} O_{K,\mathfrak{q}} \cong \Hom_{O_{K_0}}(A_0[\mathfrak{q}_0^{\infty}],A[\mathfrak{q}^{\infty}])$$
imply that both domain and codomain of $F$ are free of rank 1 over $O_{K,\mathfrak{q}}$, so that $F$ is an isomorphism (clearly also an isomorphism of Hermitian spaces).
\end{proof}

Now we prove the theorem using the lemma: First if $\mathfrak{q}$ is not $\p_F$, then by lemma we have
$$L_{Betti}(A_0^{\prime},A^{\prime}) \otimes_{O_F} O_{F,\mathfrak{q}} \cong L(A_0,A) \otimes_{O_F} O_{F,\mathfrak{q}}$$
so that by proposition 4.8, we have that $L(A_0,A) \otimes_{O_F} O_{F,\mathfrak{q}} \cong O_{K,\mathfrak{q}}$ with the Hermitian form given by $\beta_{\mathfrak{q}} x \bar{y}$ with $\beta_{\mathfrak{q}} \in F_{\mathfrak{q}}^{\times}$ with $\beta_{\mathfrak{q}} O_{F,\mathfrak{q}} =  \partial_{F/F_0}^{-1} O_{F,\mathfrak{q}}$. Now suppose that $\mathfrak{q} = \p_F$, then considering $\Phi(\mathfrak{q})$ and $\Phi_0(\mathfrak{q}_0)$ as before, by proposition 3.2 gives us that $L(A_0,A) \otimes_{O_F} O_{F,\mathfrak{q}} \cong \Hom_{O_{K_0}}(A_0[\mathfrak{q}_0^\infty],A[\mathfrak{q}^{\infty}]) \cong O_{K,\mathfrak{q}}$ with Hermitian form given by $\beta_{\mathfrak{q}} x \bar{y}$ with $\beta_{\mathfrak{q}} \in F_{\mathfrak{q}}^{\times}$ with $\beta_{\mathfrak{q}} O_{F,\mathfrak{q}} = \partial_{F/F_0}^{-1} \p_F^{\epsilon_{\p_F}} O_{F,\mathfrak{q}}$.

So we have the required isomorphism as in the statement of the theorem. Now as $L(A_0,A) \otimes_{O_K} K$ is a $K$-Hermitian space, we have that $\beta$ differs from some $\beta^{\ast} \in F^{\times}$ by a norm at each place, so that $\chi(\beta) = \chi(\beta^{\ast}) = 1$ which proves the theorem.

Now we are going to define $\mathcal{Z}(\alpha)$ for $\alpha \in  F^{\gg 0}$. It is the Deligne-Mumford stack over $O_{\tilde{K}}$ such that for an $O_{\tilde{K}}$-scheme $S$ it gives us the groupoid of $(A_0,A,f)$ with $(A_0,A) \in \mathcal{X}(S)$ and $f \in L(A_0,A)$ with $\langle f,f \rangle_{CM} = \alpha$.
\begin{proposition}
(1) Let $\alpha \in  F^{\gg 0}$, the stack $\mathcal{Z}(\alpha)$ has dimension 0 and it is supported in nonzero characteristic.\\
(2) If $\p$ is a prime of $\tilde{K}$ with $\mathcal{Z}(\alpha)(\ktildebarp)$ nonempty, then $\p_F$ is nonsplit. 
\end{proposition}
\begin{proof}
(1) The forget map $\mathcal{Z}(\alpha) \rightarrow \mathcal{X}$ is unramified, so induces a surjection on the completed strictly Henselian local rings, so that if $z \in \mathcal{Z}(\alpha)(\ktildebarp)$ is a point, then $\hat{O}_{\mathcal{Z}(\alpha),z}$ is a quotient of $\tilde{W}$, so because $\mathcal{Z}(\alpha)$ does not have a point in characteristic 0 (due to the fact that signatures of $(A_0,A)$ have to be different) and has dimension 0.\\
(2) If $\p$ is a prime that $\mathcal{Z}(\alpha)(\ktildebarp)$ is nonempty, then by the signatures of $A_0,A$ we have that $\phi^1_1 = \bar{\phi}^1_1$, so that $x = \bar{x}\; \text{mod}\; \p$ for that $x \in K$ and so $\p = \bar{\p}$ and so $\p_F$ is nonsplit $K$. 
\end{proof}
Now let $C_K^0 \subseteq C_K$ be the subgroup defined by the exact sequence 
$$1 \rightarrow C_K^0 \rightarrow C_K \xrightarrow{gen} \hat{O}_F^{\times}/N_{K/F}(\hat{O}_K^{\times}) \xrightarrow{\eta} \lbrace \pm 1 \rbrace$$
where the map is the restriction of the character $\chi$ where 
$gen(I,\zeta) = \zeta z\bar{z}$ where $z \in \hat{K}^{\times}$ has the property $zO_K = I$. Now we have the following assumption for the rest of our manuscript.\\
{\bf Assumption.} We assume that $[K:K_0]$ is even and that for all primes $\p$ of $\tilde{K}$ with residue characteristic $p$, the CM abelian varieties $(A,\iota,\lambda)$ appearing in $(A_0,A) \in CM_{\Phi}(\ktildebarp)$, $(A \otimes \mathbb{Z}_p, \iota \otimes \mathbb{Z}_p, \lambda \otimes \mathbb{Z}_p)$ satisfies the conditions in chapter 4 of \cite{rsz} (it is sufficient to assume this for primes $\p$ such that $\p_F$ is nonsplit in $K$ and only the conditions happening in section 4.4 of \cite{rsz} as the other conditions are satisfied).

Assuming the assumption above, for each place $v$ of $F_0$, choose $W_v$ in a way that there exists at least one $(A_0,A)$ in each of $\mathcal{X}(\ktildebarp)$. We see that we have exactly one $C_K^0$-orbit in each $\mathcal{X}(\ktildebarp)$ (because the sign conditions of $\mathcal{X}$ implies that there's exactly one genus of Hermitian spaces in each fiber of $\mathcal{M}_0^{\xi} \otimes_{O_{\tilde{K}}} \cm_{\Phi}^{\mathfrak{a}}$) which by \cite{howard} page 1137, $C_K^0$ acts simply transitively on. Now we compute the number of stacky points of $\mathcal{Z}(\alpha)(\ktildebarp)$. Let $w(K), w(K_0)$ be the number of roots of unity in $K, K_0$, respectively.
\begin{theorem}
Suppose that $\alpha \in F$ and $\alpha \gg 0$. Also let $\beta$ be the one appearing in Theorem 4.9. If $\p$ is a prime of $\tilde{K}$ with $\p_F$ nonsplit in $K$, then 
$$\sum_{(A_0,A,f) \in \mathcal{Z}(\alpha)(\ktildebarp)} \frac{1}{\# Aut(A_0,A,f)} = \frac{1}{w(K_0)} \rho(\alpha \partial_{F/F_0} \mathfrak{a}^{-1} \p_F^{-\epsilon_{\p_F}})$$
if $\alpha \beta \in N_{\hat{K}/\hat{F}}(\hat{K}^{\times})$ and $0$ if not.
\end{theorem}
\begin{proof}
We have 
$$\sum_{I \in C_K^0} \#\lbrace f \in L(A_0,A^I)| \langle f,f \rangle_{CM}^I = \alpha  \rbrace = \sum_{I \in C_K^0} \sum_{\substack{x \in L(A_0,A) \otimes \mathbb{Q}\\ \langle x,x \rangle = \alpha}} 1_{I.L(A_0,A)}(x)$$
where $1_A$ is the characteristic function of $A$. Now using the presentation $C_K^0 = H(F)\backslash H(\hat{F})/U$ using the algebraic group $H$ and $U$ defined in \cite{howard}. The sum above is equal to
$$\sum_{h \in H(F)\backslash H(\hat{F})/U} \sum_{\substack{x \in V(A_0,A)\\ \langle x,x \rangle_{CM} = \alpha}} 1_{\hat{L}(A_0,A)(h^{-1}x)} = \#(H(F) \cap U)\sum_{h \in H(\hat{F})/U} \sum_{\substack{x \in H(F)\backslash V(A_0,A) \\ \langle x,x \rangle_{CM} = \alpha}} 1_{\hat{L}(A_0,A)}(h^{-1}x)$$
Let $\mu(K_0), \mu(K)$ be the group of roots of unity of $K_0, K$ respectively. Now we have that $H(F) \cap U = \mu(K)$ and also $Aut(A_0,A) \cong \mu(K_0) \times \mu(K)$. So we get that 
\begin{equation}
\sum_{I \in C_K^0} \sum_{\substack{f \in L(A_0,A^I) \\ \langle f,f \rangle^I_{CM} = \alpha}} \frac{w(K_0)}{\# Aut(A_0,A^I)} = 
\sum_{h \in H(\hat{F})/U} \sum_{\substack{x \in H(F)\backslash V(A_0,A) \\ \langle x,x \rangle_{CM} = \alpha}} 1_{\hat{L}(A_0,A)}(h^{-1}x)
\end{equation}
Now there are two cases, either there is an $x \in V(A_0,A)$ with $\langle x,x \rangle_{CM} = \alpha$ or there is no $x$ with $\langle x,x \rangle_{CM} = \alpha$. In the latter case, the RHS is zero and in the former case $H(F)$ acts simply transitively on them and so the RHS is $$\frac{1}{w(K_0)} \sum_{h \in H(\hat{F})/ U} 1_{\hat{L}(A_0,A)}(h^{-1}x)$$
Now we define the orbital integral for $\alpha \in F^{\times}$ by
$$O_{\alpha}(A_0,A) = \sum_{h \in H(\hat{F})/U} 1_{\hat{L}(A_0,A)}(h^{-1}x)$$
where $x \in \hat{V}(A_0,A)$ has the property $\langle x,x \rangle_{CM} = \alpha$. If such an $x$ does not exist, $O_{\alpha}(A_0,A)$ is defined to be zero. Now the RHS of equation 4.1 is $\frac{1}{w(K_0)} O_{\alpha}(A_0,A)$. Now let $\beta$ be the element of $\hat{F}^{\times}$ such that the Hermitian form on $\hat{L}(A_0,A)$ is $\beta x \bar{y}$. We break the orbital integral into local parts:
$$O_{\alpha}(A_0,A) = \prod_{v} O_{\alpha,v}(A_0,A)$$
where $O_{\alpha,v}(A_0,A) = \sum_{h \in H(F_v)/U_v} 1_{O_{K,v}}(h^{-1}x_v)$.
We now have two cases, either $v$ is nonsplit in $K$ and we will get 
\begin{equation}
O_{\alpha,v}(A_0,A) =
    \begin{cases*}
      1 & if $\alpha\beta^{-1} \in O_{F,v}$ \\
      0        & otherwise
    \end{cases*}
\end{equation}
or in the split case we see that (in the same way as in \cite{cheraghi}):
\begin{equation}
O_{\alpha,v}(A_0,A) =
    \begin{cases*}
      1+\ord_v(\alpha_v\beta_v^{-1}) & if $\alpha\beta^{-1} \in O_{F,v}$ \\
      0       & otherwise
    \end{cases*} 
\end{equation}
so the product above is going to be $\rho(\alpha\beta^{-1}
O_F) = \#\lbrace J \triangleleft O_K | N_{K/F}J = \alpha\beta^{-1}O_F \rbrace$ and now if $\rho(\alpha \beta^{-1}O_F) \neq 0$ then $\alpha \beta^{-1} \in N_{\hat{K}/\hat{F}}(\hat{K}^{\times})$ and using the fact that we have the ideal of $\beta O_F$, if $\alpha \beta^{-1} \not\in N_{\hat{K}/\hat{F}}(\hat{K}^{\times})$, then $\rho(\alpha \beta^{-1}O_F) = 0$ and we get $O_{\alpha}(A_0,A) = 0$, 
So we finally get the statement of the theorem.
\end{proof}
Now we need a theorem about lengths of strictly henselian local rings:
\begin{theorem}
Let $\alpha \in F^{\times}$ and $\p$ a prime of $\tilde{K}$ such that $\p_F$ is nonsplit in $K$. Then at a point $z \in \mathcal{Z}(\alpha)(\ktildebarp)$, we have 
$$\text{length}(O^{s.h.}_{\mathcal{Z}(\alpha),z}) = \frac{1}{2} \ord_{\tilde{K}_{\p}}(\alpha \p_F\mathfrak{a}^{-1}\partial_{F/F_0})$$
\end{theorem}
\begin{proof}
Consider $(A_0,A,f) \in \mathcal{Z}(\alpha)(\ktildebarp)$ be the triple corresponding to $z$, then the completed strictly henselian ring $\hat{O}^{s.h.}_{\mathcal{Z}(\alpha),z}$ pro-represents the deformations of $(A_0,A,f)$ to objects of ART which by Serre-Tate, is in turn 
the same as deformations of $(A_0[p^{\infty}],A[p^{\infty}],f[p^{\infty}])$ to the objects 
of ART. Now we have the decomposition $A_0[p^{\infty}] = \prod_{\substack{\mathfrak{q}_0|p \\ \mathfrak{q}_0 \triangleleft O_{F_0}}} A_0[\mathfrak{q}^{\infty}]$ and $A[p^{\infty}] = 
\prod_{\substack{\mathfrak{q}|p \\ \mathfrak{q} \triangleleft O_F}} A[\mathfrak{q}^{\infty}]
$, so that the map $f[p^{\infty}]$ is decomposed into $f_{\mathfrak{q}_0,\mathfrak{q}}: 
A_0[\mathfrak{q}_0^{\infty}] \rightarrow A[\mathfrak{q}^{\infty}]$ for different $
\mathfrak{q}_0 \triangleleft O_{F_0}$ and $\mathfrak{q} \triangleleft O_F$ above the prime $p$, so we have to analyze the liftings of $f[\mathfrak{q}_0^{\infty}]$ to higher Artin rings (i.e. to higher powers $k$ in $\tilde{W}/m^k$). Now we have two cases:\\
(1) $\mathfrak{q} \neq \p_F$, in this case, the $p$-adic CM-types of $A_0[\mathfrak{q}_0^{\infty}]$ and $A[\mathfrak{q}^{\infty}]$ are compatible (i.e. the $p$-adic CM-type of $A[\mathfrak{q}^{\infty}]$ is exactly the embeddings whose restriction to $F_{0,\mathfrak{q}_0}$ induces the embeddings in the $p$-adic CM-types of $A_0[\mathfrak{q}_0^{\infty}]$).\\
(2) $\mathfrak{q} = \p_F$, in this case, the $p$-adic CM-types of $A_0[\mathfrak{q}_0^{\infty}]$ and $A[\mathfrak{q}^{\infty}]$ are incompatible and there's exactly one embedding in the $p$-adic CM-type of $A[\mathfrak{q}^{\infty}]$ such that restriction to $F_{0,\mathfrak{q}_0}$ is the conjugation of one embedding of $p$-adic CM-type of $A_0[\mathfrak{q}_0^{\infty}]$, so we are in the situation of theorem 3.5 in section 3 and so the deformations of $(A_0[\mathfrak{q}_0^\infty], A[\mathfrak{q}^{\infty}], f_{\mathfrak{q}_0,\mathfrak{q}})$ to objects of ART is pro-represented by $\tilde{W}/m^k$ where $k = \frac{1}{2} \ord_{\tilde{K}_{\p}}(\alpha \mathfrak{a}^{-1} \p_F \partial_{F/F_0}) = \frac{1}{2} \ord_{\tilde{K}_{\p}}(\alpha \mathfrak{a}^{-1} \p_F \partial_{F/F_0})$
So we get 
$$\text{length}(O^{s.h.}_{\mathcal{Z}(\alpha),z}) = \text{length}(\hat{O}^{s.h.}_{\mathcal{Z}(\alpha),z}) = \frac{1}{2}\ord_{\tilde{K}_{\p}}(\alpha \mathfrak{a}^{-1} \p_F \partial_{F/F_0})$$
\end{proof}
Now we collect everything from theorem 4.12 and 4.13 and we get the main result of this section:
\begin{theorem}
If $\alpha \in F^{\gg 0}$ and for $\p$ a prime of $\tilde{K}$, let $\beta_{\p}$ be the $\beta$ appearing in Theorem 4.9, then 
$$\widehat{\deg} \mathcal{Z}(\alpha) = \frac{1}{2w(K_0)} \sum_{\substack{\p \\ \alpha \beta_{\p}\in N_{\hat{K}/\hat{F}}(\hat{K}^{\times})}} \rho(\alpha \mathfrak{a}^{-1} \partial_{F/F_0} \p_F^{-\epsilon_{\p_F}}) \ord_{\tilde{K}_{\p}}(\alpha \mathfrak{a}^{-1} \p_F \partial_{F/F_0})\frac{\log N_{\tilde{K}/\mathbb{Q}}(\p)}{[\tilde{K}:\mathbb{Q}]} =$$
$$= \frac{1}{w(K_0)}\sum_{\substack{\mathfrak{q} \subseteq O_F \\ \alpha \beta_{\mathfrak{q}} \in N_{\hat{K}/\hat{F}}(\hat{K}^{\times})}} \rho(\alpha \mathfrak{a}^{-1} \partial_{F/F_0} \mathfrak{q}^{-\epsilon_{\mathfrak{q}}}) \ord_{\mathfrak{q}}(\alpha \mathfrak{a}^{-1} \mathfrak{q} \partial_{F/F_0})\frac{\log N_{F/\mathbb{Q}}(\mathfrak{q})}{[K:\mathbb{Q}]}$$
where $\mathfrak{q}$ changes over the primes of $O_F$ nonsplit in $K$ and $\p$ appearing in $\beta_{\p}$ in the second sum is a choice of prime $\p$ of $\tilde{K}$ above $\mathfrak{q}$.
\end{theorem}
\begin{proof}
This results from theorem 4.12 and 4.13.
\end{proof}

\section{Arithmetic Chow group}
In this section, we are going to define the arithmetic divisors as elements of $\widehat{\text{CH}}^1(\mathcal{X})$ and find their degrees. These degrees will in turn be related to not positive definite coefficients of the Eisenstein series that we are going to define later (see section 5).\\
An arithmetic divisor of $\mathcal{X} = \mathcal{M}_0^{\xi} \times_{O_{\tilde{K}}} \cm_{\Phi}^{\mathfrak{a}}$  is a pair $(Z,\text{Gr})$ such that $Z$ is a Weil divisor on $
\mathcal{X}$ and $\text{Gr}$ is a Green function for $Z$. We are going to define the 
arithmetic divisors $\widehat{\mathcal{Z}}(\alpha)$ for $0 \neq \alpha \in F^{\times}$ that are not necessarily totally positive. If $\alpha \gg 0$, we want to get $\widehat{\mathcal{Z}}(\alpha) = (\mathcal{Z}(\alpha),0)$, and in the other cases we want to get $\widehat{\mathcal{Z}}(\alpha) = (0,\text{Gr}_\alpha)$ for some Green function $\text{Gr}_\alpha$.\\
As $\text{Gr}_{\alpha}$ is a Green function for $\widehat{\mathcal{Z}}(\alpha)$ and $\widehat{\mathcal{Z}}(\alpha)$ does not have any characteristic zero points, we have that $\text{Gr}_\alpha$ can be any complex-valued function on the finite set
$$\coprod_{\substack{\sigma: \tilde{K} \rightarrow \mathbb{C}\\ \sigma|_{K_0} \in \Phi_0}}((\mathcal{M}_0^{\xi}) \times_{\O_{\tilde{K}}} \cm_{\Phi}^{\mathfrak{a}})^{\sigma}(\mathbb{C})$$
Now we define the Green functions $\text{Gr}_\alpha$ on the point $z \in ((\mathcal{M}_0^{\xi}) \times_{\O_{\tilde{K}}} \cm_{\Phi}^{\mathfrak{a}})^{\sigma}(\mathbb{C})$ corresponding to $(A_0,A)$ to be 
$$\text{Gr}_{\alpha}(y,\alpha) = \sum_{\substack{f \in L_{Betti}(A_0,A)\\ \langle f,f \rangle_{CM} = \alpha}} \beta_1(4\pi |y\alpha|_{\tilde{\sigma} \circ \infty^{sp}})$$
where $\tilde{\sigma} \in \text{Aut}(\mathbb{C})$ is an extension of $\sigma: \tilde{K} \rightarrow \mathbb{C}$ and $\beta_1(t)  = \int_1^\infty e^{-tu} \frac{du}{u}$.\\
We have that $\widehat{\mathcal{Z}}(\alpha) = (\mathcal{Z}(\alpha), \text{Gr}_\alpha)$ is an element of the first Chow group $\widehat{\text{CH}}^1(\mathcal{X})$ and on this Chow group, we have a degree function that maps:
$$\widehat{deg}: \widehat{\text{CH}}^1(\mathcal{X}) \rightarrow \widehat{\text{CH}}^1(Spec\; O_{\tilde{K}}) \rightarrow \mathbb{R}$$
and for an arithmetic divisor $(Z,\text{Gr})$, it is defined to be
$$\widehat{\deg}(Z,\text{Gr}) = \frac{1}{[\tilde{K}:\mathbb{Q}]} (\sum_{\p \in O_{\tilde{K}}} \sum_{z \in Z(\ktildebarp)} \frac{\log N(\p)}{\# \text{Aut}\; z} + \sum_{\substack{\sigma:\tilde{K} \rightarrow \mathbb{C} \\ \sigma|_{K_0} \in \Phi_0}}\sum_{z \in ((\mathcal{M}_0^{\xi}) \times_{\O_{\tilde{K}}} \cm_{\Phi}^{\mathfrak{a}})^{\sigma}(\mathbb{C})} \frac{\text{Gr}(z)}{\# \text{Aut}\; z})$$
Now an easy computation for $\alpha$ not positive definite shows that
$$\widehat{\deg}\widehat{\mathcal{Z}}(\alpha) = \begin{cases*}
      \frac{1}{w(K_0) [K:\mathbb{Q}]} \beta_1(4\pi|y\alpha|_v) \rho(\alpha \partial_{F/F_0} 
      \mathfrak{a}^{-1}) &  \text{if} $\alpha$ \text{is negative definite at exactly one place} \\
      0        & \text{otherwise}
    \end{cases*}$$

\section{Eisenstein series}
For completeness we define the Eisenstein series in this section. These are Eisenstein series with the property that the Fourier coefficients of this Eisenstein series are related to the degree of divisors considered in previous sections.\\
Let the notations be as in the notations section. Fix a place $v$ of $K$ and let $v_F$ be the prime below $v$ in $F$ and some $c \in F_{v_F}^{\times}$. $\chi_{v_F}: F_{v_F}^{\times} \rightarrow \mathbb{C}^{\times}$ be character of $K_v/F_{v_F}$. $\psi$ be an additive character $F_{v_F} \rightarrow \mathbb{C}^{\times}$. Now there's a space of Schwartz functions $G(K_v)$ and $L(\chi_v, s)$ the space of induced representation of $\chi_v(x)|x|_v^{s}$. These two spaces have actions of $SL_2(F_{v_F})$ and we have an operator 
$$\lambda_{c,\psi}: G(K_v) \rightarrow L(\chi_v,0)$$
$$\lambda_{c,\psi}(\phi)(g) = (\omega_{c,\psi}(g)\phi)(0)$$ 
There's a unique section $\Phi_{c,\psi}(g,s) \in I(\chi_v,s)$ with $\Phi(.,0) = \lambda_{c,\psi}(1_{O_{K_v}})$ ($1_{O_{K_v}}$ is the characteristic function of $O_{K_v}$ in $K_v$) and $\Phi(g,s)$ is independent of $s$ for $g$ in maximal compact subgroup of $SL_2(F_{v_F})$ if $v$ is nonarchimedean. $\Phi(.,0) = \lambda_{c,\psi}(e^{-2\pi |cx\bar{x}|_v})$ if $v$ in archimedean. For $\alpha \in F^{\times}_{v_F}$, define the local Whittaker function
$$W_{\alpha}(g,s,\Phi_{c,\psi},\psi) = \int_{F_{v_F}} \Phi_{c,\psi}(\begin{bmatrix}
0 & -1 \\
1 & 0 
\end{bmatrix} \begin{bmatrix}
1 & x \\
0 & 1
\end{bmatrix}g,s)\psi_v(-\alpha x)dx$$
Now I want to define the setup for global situation. Let $\psi_{\mathbb{Q}}: \mathbb{A}_{\mathbb{Q}}/\mathbb{Q} \rightarrow \mathbb{C}^{\times}$ be the additive character with $
\psi_{\mathbb{Q}}(x) = e^{2\pi i x}$ for $x \in \mathbb{R}$ and unramified nonarchimedean components (i.e. $\psi_{\mathbb{Q}}(\bar{\mathbb{Z}_p}) = 1$ for all $p$ where $\mathbb{Z}_p
$ is $\mathbb{Z}_p\mathbb{Q}/\mathbb{Q} \subseteq \mathbb{A}_{\mathbb{Q}}/\mathbb{Q}$). Let $\psi_F(x) = \psi_{\mathbb{Q}}(tr_{F/\mathbb{Q}}(x))$ and $\chi: \mathbb{A}_F^{\times} 
\rightarrow \mathbb{C}^{\times}$ be the character of $K/F$. For $c \in \mathbb{A}_F^{\times}$, let $\Phi_{c,\psi_F} = \otimes_v \Phi_{c,\psi_{F_v}}$ and define an Eisenstein series 
$$E(g,s,c,\psi_F) = \sum_{\gamma \in B(F)\backslash SL_2(F)} \Phi_{c,\psi_F}(\gamma g,s)$$
($B$ is the Borel subgroup of $SL_2$ of upper-triangular matrices). For normalizing the above Eisenstein series, let $\mathbb{H}_F = \lbrace x+iy \in F\otimes_{\mathbb{Q}} \mathbb{C}| x,y \in F\otimes_{\mathbb{Q}} \mathbb{R}, y \gg 0 \rbrace$. For $\tau = x+iy$, let $g_{\tau} \in SL_2(\mathbb{A}_F)$ have archimedean components
$$\begin{bmatrix}
1 & x\\
0 & 1
\end{bmatrix}\begin{bmatrix}
y^{\frac{1}{2}} & 0\\
0 & y^{-\frac{1}{2}} 
\end{bmatrix} \in SL_2(F \otimes_{\mathbb{Q}} \mathbb{R})$$ and trivial components. Now let the normalized Eisenstein series be (by abuse of notation):
$$E(\tau,s,c,\psi_F) = N_{F/\mathbb{Q}}(\partial_{F/F_0})^{\frac{s+1}{2}} \frac{L(s+1,\chi)}{N_{F/\mathbb{Q}}(y)^{\frac{1}{2}}} E(g_{\tau}, s,c,\psi_F)$$
where $L(s,\chi)$ is the Dirichlet function of $\chi$. $E$ has a Fourier expansion
$$E(\tau, s,c,\psi_F) = \sum_{\alpha \in F} E_{\alpha}(\tau,s,c,\psi_F)$$
with $$E_{\alpha}(\tau, s, c, \psi_F) = N_{F/\mathbb{Q}}(y)^{-\frac{1}{2}} \int_{F \mathbb{A}_F}E(\begin{bmatrix}
1 & b\\
0 & 1
\end{bmatrix}g_{\tau},s,c,\psi_F(-ba))db$$
Now let $c$ has the property $cO_F = \partial_{F/F_0}^{-1} \mathfrak{a}$ and $c$ has trivial archimedean components with $\chi(c) = -1$. $\chi(c) = -1$ implies that the Eisenstein series is incoherent and so $E(\tau,0,c,\psi_F) = 0$.\\
We finally define the Eisenstein series to be 
$$\mathcal{E}(\tau,s) = {E}(\tau,s,c,\psi_F)$$
where $c \in \mathbb{A}_{F}^{\times}$ is taking $\beta$ appearing in prop 4.8 and replacing the component at $\infty^{sp}$ with a positive definite element. This Eisenstein series has Fourier coefficients $\mathcal{E}(\tau,s) = \sum_{\alpha \in F} \mathcal{E}_{\alpha}(\tau,s)$. The theorem below computes the Fourier coefficients of the derivative of $\mathcal{E}(\tau,s)$ at $s=0$ and also the order of Fourier coefficients of $\mathcal{E}(\tau,s)$ at $s=0$. Let $\text{Diff}(\alpha,c) = \lbrace v|\chi_v(\alpha c) = -1 \rbrace$ be a finite subset of places of $F$. This set is easily seen to have odd cardinality by $\chi(c) = -1$.
\begin{theorem}
Let $\alpha$ be nonzero $F$ and $d_{K/F}$ be the relative discriminant of $K/F$, $r$ be the number of places of $F$ ramified in $K$ (including the archimedean places). Then
\begin{itemize}
\item $\# \text{Diff}(\alpha,c) > 1 \Rightarrow \ord_{s=0}({E}_{\alpha}(\tau,s,c,\psi_F)) > 1$
\item If $ \text{Diff}(\alpha,c) = \lbrace \p \rbrace$ with $\p$ a finite prime of $F$, then
$${E}^{\prime}_{\alpha}(\tau,0,c,\psi_F) = \frac{-2^{r-1}}{N_{F/\mathbb{Q}}(d_{K/F})^{\frac{1}{2}}} \rho(\alpha \partial_{F/F_0} \mathfrak{a}^{-1} \p^{-\epsilon_{\p}}) \ord_{\p}(\alpha \partial_{F/F_0} \mathfrak{a}^{-1} \p) \log(N_{F/\mathbb{Q}}(\p)) e^{2\pi i tr_{F/\mathbb{Q}}(\alpha \tau)}$$
\item If $\text{Diff}(\alpha,c) = \lbrace w \rbrace$ where $w$ is an archimedean place, then
$$E_{\alpha}^{\prime}(\tau, 0, c, \psi_F) = \frac{2^{-(r-1)}}{N(d_{K/F})^{\frac{1}{2}}} \rho(\alpha\partial_{F/F_0}\mathfrak{a}^{-1})\beta_1(4\pi|y\alpha|_v)q^\alpha$$
\end{itemize}
\end{theorem}
\begin{proof}
For $g \in SL_2(F_v)$, consider the normalized local Whittaker function
$$W_{\alpha_v}^{\ast}(g_v,s,c_v,\psi_{F_v}) = L(s+1,\chi_v) W_{\alpha_v}(g_v,s,c_v,\psi_{F_v})$$
Now we have the factorization
$${E}_{\alpha}(\tau,s,c,\psi_F) = N_{F/\mathbb{Q}}(y)^{-\frac{1}{2}} \prod_v W_{\alpha_v}^{\ast}(g_{\tau,v},s,c_v,\psi_{F_v})$$
so we have
$${E}_{\alpha}(\tau,s,c,\psi_F) = N_{F/\mathbb{Q}}(y)^{-\frac{1}{2}} \prod_v W_{c_v^{-1}\alpha_v}^{\ast}(g_{\tau,v}, s,1,c_v\psi_{F_v})$$
where $(c\psi_{F})(x) = \psi_F(cx)$ is an unramified character of $\mathbb{A}_F^{\times}$ and also $(c_v\psi_{F_v})(x) = \psi_{F_v}(c_v x)$. By Yang's formula
$$\chi_v(\alpha c) = -1 \Leftrightarrow W^{\ast}_{c_v^{-1}\alpha_v}(g_{\tau,v},0,1,c_v\psi_{F_v}) = 0$$
If $v$ is nonarchimedean then by Yang's formulas \cite{yang}, we get:\\
(1) If $\chi_v(\alpha c)= 1$, then 
$$W_{c_v^{-1}\alpha}^{\ast}(g_{\tau,v},0,1,c_v\psi_{F_v}) = \chi_v(-1) \epsilon(\frac{1}{2}, \chi_v, c_v\psi_{F_v})\rho(\alpha \partial_{F/F_0}\mathfrak{a}^{-1})\begin{cases*}
      2N(\pi_{F_v})^\frac{-\ord_v(d_{K/F})}{2} & \text{if} v \text{is ramified in} K/F \\
      1        & \text{if v is unramified in K/F}
    \end{cases*}$$

(2) If $\chi_v(\alpha c) = -1$, then 
$$\frac{d}{ds} W^{\ast}_{c_v^{-1}\alpha}(g_{\tau,v}, s,1,c_v\psi_{F_v})|_{s=0} =$$
$$\chi_v(-1)\epsilon(\frac{1}{2}, \chi_v, c_v\psi_{F_v})\log|\pi_{F_v}|^{-1}\frac{\ord_v(\alpha)+1}{2}\begin{cases*}
      2N(\pi_{F_v})^{\frac{-\ord_v(d_{K/F})}{2}}\rho(\alpha \partial_{F/F_0} \mathfrak{a}^{-1}O_{F_v}) & \text{if v is ramified in K/F} \\
      \rho_v(\alpha \partial_{F/F_0} \mathfrak{a}^{-1} \p_v^{-1})        & \text{if v unramified in K/F}     \end{cases*}$$
      
If $v$ is archimedean, if $\chi_v(\alpha c) = 1$, then
$$W^{\ast}_{c_v^{-1}\alpha_v}(g_{\tau,v},0,1,c_v\psi_{F_v}) = 2\chi_v(-1)\epsilon(\frac{1}{2},\chi_v,c_v\psi_{F_v})y_v^{\frac{1}{2}}e^{2\pi i \alpha_v \tau_v}$$
and if $\chi_v(\alpha c) = -1$, then $W^{\ast}_{c_v\alpha_v^{-1}}(g_{\tau,v},0,1,c_v\psi_{F_v})y_v^{\frac{1}{2}}\beta_1(4\pi|y\alpha|_v)e^{2\pi i \alpha_v \tau_v}$      
      
Now if $\text{Diff}(\alpha,c) = \lbrace w \rbrace$, then
$$\frac{d}{ds} {E}_{\alpha}(\tau, s,c,\psi_F)|_{s=0} = N_{F/\mathbb{Q}}(y)^{\frac{-1}{2}} \frac{d}{ds}W^{\ast}_{c_w^{-1}\alpha_w}(g_{\tau,w},s,1,c\psi_{F_w})|_{s=0} \prod_{v\neq w} W^{\ast}_{c_v^{-1}\alpha_v}(g_{\tau,v}, 0,1,c\psi_{F_v})$$      
and we get the formulas stated in the statement of theorem.
\end{proof}
So we get that for $\alpha \gg 0$ in $F$:
$$\mathcal{E}_{\alpha}^{\prime}(\tau,0) = {E}_{\alpha}^{\prime}(\tau,0,c,\psi_F) =$$ $$= \frac{-2^{r-1}}{\sqrt{N_{F/\mathbb{Q}}(d_{K/F})}} \rho(\alpha \partial_{F/F_0} \mathfrak{a}^{-1} \p^{-\epsilon_{\p}})\ord_\p(\alpha \partial_{F/F_0} \mathfrak{a}^{-1})\log(N(\p))e^{2\pi i tr_{F/\mathbb{Q}}(\alpha \tau)} =: b_{\Phi}(\alpha,y)$$
for $\p \subseteq O_F$ nonsplit with $\text{Diff}(\alpha,c) = \lbrace \p \rbrace$. Now using the above and Theorem 4.14, one can check that we get the main result:
\begin{theorem}
Let $\alpha$ be nonzero in $F$. Suppose that the ramification condition in the introduction is satisfied and the assumption below proposition 4.1 is satisfied, then
$$\widehat{\deg}\mathcal{Z}(\alpha) = \frac{-1}{w(K_0)}\frac{N_{F/\mathbb{Q}}(d_{K/F})^{\frac{1}{2}}}{2^{r-1}[K:\mathbb{Q}]} b_{\Phi}(\alpha,y).$$
\end{theorem}



\begin{bibdiv}
\begin{biblist}*{labels={alphabetic}}

\bib{cheraghi}{article}{
      title={Special Correspondences of CM Abelian Varieties and Eisenstein Series}, 
      author={Cheraghi, A.},
      year={2021},
      eprint={arXiv.2107.00542},
      archivePrefix={arXiv},
      primaryClass={math.NT}
}

\bib{howard}{article}{
   author={Howard, B.},
   title={Complex multiplication cycles and Kudla-Rapoport divisors},
   journal={Ann. of Math. (2)},
   volume={176},
   date={2012},
   number={2},
   pages={1097--1171},
   issn={0003-486X},
}

\bib{rsz}{article}{
   author={Rapoport, M.},
   author={Smithling, B.},
   author={Zhang, W.},
   title={Arithmetic diagonal cycles on unitary Shimura varieties},
   journal={Compos. Math.},
   volume={156},
   date={2020},
   number={9},
   pages={1745--1824},
   issn={0010-437X},
   review={\MR{4167594}},
   doi={10.1112/s0010437x20007289},
}

\bib{yang}{article}{
   author={Yang, T.},
   title={CM number fields and modular forms},
   journal={Pure Appl. Math. Q.},
   volume={1},
   date={2005},
   number={2, Special Issue: In memory of Armand Borel},
   pages={305--340},
   issn={1558-8599},
}


\end{biblist}
\end{bibdiv}
\end{document}